\documentclass[a4paper,12pt]{amsart}
\usepackage{amsrefs}
\usepackage{mathtools}
\usepackage{amsmath,amssymb,amsthm}
\usepackage{amsaddr}
\usepackage{amsfonts}
\usepackage{graphicx}
\usepackage{braket}
\usepackage{chngcntr}
\usepackage{subcaption}
\counterwithin{figure}{section}
\usepackage[T1]{fontenc}
\usepackage{hyperref}
\usepackage{tikz}
\usepackage{graphicx}
\usepackage{mathtools}
\usetikzlibrary{arrows,shapes,positioning}
\usetikzlibrary{decorations.markings}
\tikzstyle arrowstyle=[scale=1]
\tikzstyle directed=[postaction={decorate,decoration={markings,
	mark=at position .55 with {\arrow[arrowstyle]{stealth}}}}]
\newtheorem{thm}{Theorem}
\newtheorem{lem}{Lemma}
\newtheorem{cor}{Corollary}
\newtheorem{defn}{Definition}
\newtheorem{prop}{Proposition}

\theoremstyle{definition}
\newtheorem{defn2}{Example}

\newcommand{\Complex}{\mathbb{C}}

\newcommand{\Real}{\mathbb{R}}
\newcommand{\tree}{%
	\vcenter{\hbox{\tikz[node distance=2.5ex]{%
				\draw[thick] (5,-0.20) -- (5,0) -- (4.85,0.10) -- (5,0) -- (5.15,0.10);
}}}}
\newcommand{\bridge}{\overset{\frown}{\vee}}

\newcommand{\oneloop}{%
	\vcenter{\hbox{\tikz[node distance=2.5ex]{%
				\draw[thick] (5,-0.20) -- (5,0) -- (4.85,0.10) -- (5,0) -- (5.15,0.10) ;\draw[thick] (5.15,0.13) arc (45:135:0.2cm);
}}}}
\newcommand{\onetwo}{%
	\vcenter{\hbox{\tikz[node distance=2.5ex]{%
				\draw[thick] (6.25,0.75) -- (6.25,1) -- (6,1.25) -- (6.125,1.125)--(6.25,1.25)--(6.125,1.125) -- (6.25,1) -- (6.5,1.25);
}}}}
\newcommand{\twoone}{%
	\vcenter{\hbox{\tikz[node distance=2.5ex]{%
				\draw[thick] (6.25,0.75) -- (6.25,1) -- (6,1.25) -- (6.125,1.125) -- (6.25,1) -- (6.375,1.125)--(6.25,1.25)--(6.375,1.125)--(6.5,1.25);
}}}}
\newcommand{\onetwoloopone}{%
	\vcenter{\hbox{\tikz[node distance=2.5ex]{%
				\draw[thick] (6.25,0.75) -- (6.25,1) -- (6,1.25) -- (6.125,1.125)--(6.25,1.25)--(6.125,1.125) -- (6.25,1) -- (6.5,1.25);\draw[thick] (6.25,1.3) arc (45:135:0.18cm);
}}}}
\newcommand{\onetwolooponeirreg}{%
	\vcenter{\hbox{\tikz[node distance=2.5ex]{%
				\draw[thick] (6.25,0.75) -- (6.25,1) -- (6,1.25) -- (6.125,1.125)--(6.25,1.25)--(6.125,1.125) -- (6.25,1) -- (6.5,1.25);\draw[thick] (6.25,1.3) arc (45:135:0.18cm);\draw[thick] (6.5,1.3) arc (45:135:0.18cm);
}}}}
\newcommand{\onetwolooptwo}{%
	\vcenter{\hbox{\tikz[node distance=2.5ex]{%
				\draw[thick] (6.25,0.75) -- (6.25,1) -- (6,1.25) -- (6.125,1.125)--(6.25,1.25)--(6.125,1.125) -- (6.25,1) -- (6.5,1.25);\draw[thick] (6.5,1.3) arc (45:135:0.15cm);
}}}}
\newcommand{\twooneloopone}{%
	\vcenter{\hbox{\tikz[node distance=2.5ex]{%
				\draw[thick] (6.25,0.75) -- (6.25,1) -- (6,1.25) -- (6.125,1.125) -- (6.25,1) -- (6.375,1.125)--(6.25,1.25)--(6.375,1.125)--(6.5,1.25);\draw[thick] (6.23,1.3) arc (45:135:0.17cm);
}}}}
\newcommand{\twoonelooponeirreg}{%
	\vcenter{\hbox{\tikz[node distance=2.5ex]{%
			\draw[thick] (6.25,0.75) -- (6.25,1) -- (6,1.25) -- (6.125,1.125) -- (6.25,1) -- (6.375,1.125)--(6.25,1.25)--(6.375,1.125)--(6.5,1.25);\draw[thick] (6.23,1.3) arc (45:135:0.17cm);\draw[thick] (6.5,1.3) arc (45:135:0.17cm);
}}}}
\newcommand{\twoonelooptwo}{%
	\vcenter{\hbox{\tikz[node distance=2.5ex]{%
				\draw[thick] (6.25,0.75) -- (6.25,1) -- (6,1.25) -- (6.125,1.125) -- (6.25,1) -- (6.375,1.125)--(6.25,1.25)--(6.375,1.125)--(6.5,1.25);\draw[thick] (6.5,1.3) arc (45:135:0.15cm);
}}}}
\newcommand{\threeoneloopone}{%
	\vcenter{\hbox{\tikz[node distance=2.5ex]{%
			\draw[thick] (4,0.5) -- (4,1) -- (3.5,1.5) -- (4,1)--(4.5,1.5)--(4.375,1.375)--(4.25,1.5)--(4.375,1.375)--(4.25,1.25)--(4,1.5);
			\draw[thick] (4.5,1.55) arc (30:150:0.15cm);
}}}}
\newcommand{\threeonelooptwo}{%
	\vcenter{\hbox{\tikz[node distance=2.5ex]{%
				\draw[thick] (1,0.5) -- (1,1) -- (0.5,1.5) -- (0.625,1.375)--(0.75,1.5)--(0.625,1.375)-- (1,1) --(1.375,1.375)--(1.25,1.5)--(1.375,1.375) -- (1.5,1.5);
				\draw[thick] (1.5,1.55) arc (30:150:0.15cm);
}}}}
\newcommand{\threetwoloopone}{%
	\vcenter{\hbox{\tikz[node distance=2.5ex]{%
				\draw[thick] (7,0.5) -- (7,1) -- (6.5,1.5) -- (7,1) -- (7.5,1.5)--(7.25,1.25)--(7.125,1.375)--(7.25,1.5)--(7.125,1.375)--(7,1.5);\draw[thick] (6.95,1.55) arc (30:150:0.27cm);
				\draw[thick] (7.5,1.55) arc (30:150:0.13cm);
}}}}
\newcommand{\threetwolooptwo}{%
	\vcenter{\hbox{\tikz[node distance=2.5ex]{%
				\draw[thick] (7,0.5) -- (7,1) -- (6.5,1.5) -- (6.75,1.25)--(7,1.5)--(6.875,1.375)--(6.75,1.5)-- (6.875,1.375) --(6.75,1.25)-- (7,1)-- (7.5,1.5);
				\draw[thick] (7.5,1.55) arc (45:135:0.36cm);\draw[thick] (6.75,1.55) arc (30:150:0.15cm);
}}}}
\newcommand{\threetwoloopthree}{%
	\vcenter{\hbox{\tikz[node distance=2.5ex]{%
					\draw[thick] (4,0.5) -- (4,1) -- (3.5,1.5) -- (4,1) -- (4.5,1.5)--(4.375,1.375)--(4.25,1.5)--(4.375,1.375)--(4.25,1.25)--(4,1.5);
				\draw[thick] (3.95,1.55) arc (30:150:0.27cm);\draw[thick] (4.5,1.55) arc (30:150:0.13cm);
}}}}
\newcommand{\fourtwooloopone}{%
	\vcenter{\hbox{\tikz[node distance=2.5ex]{%
			\draw[thick] (1,0.5) -- (1,1) -- (0.5,1.5)--(0.75,1.25)--(1,1.5)--(0.875,1.375)--(0.75,1.5) --(0.875,1.375)--(0.75,1.25)--(1,1)--(1.375,1.375)--(1.25,1.5)--(1.375,1.375) -- (1.5,1.5);
			\draw[thick] (0.75,1.55) arc (30:150:0.15cm);\draw[thick] (1.5,1.55) arc (30:150:0.15cm);
}}}}

\newcommand{\fourtwooloopthree}{%
	\vcenter{\hbox{\tikz[node distance=2.5ex]{%
				\draw[thick] (1,0.5) -- (1,1) --(0.5,1.5)--(1,1)--(1.375,1.375)--(1.25,1.5)--(1.375,1.375) -- (1.5,1.5)--(1.25,1.25)--(1,1.5)--(1.25,1.25)--(1.125,1.125)--(0.75,1.5);
				\draw[thick] (0.75,1.55) arc (30:150:0.15cm);\draw[thick] (1.5,1.55) arc (30:150:0.15cm);
}}}}
\newcommand{\fourtwooloopfour}{%
	\vcenter{\hbox{\tikz[node distance=2.5ex]{%
				\draw[thick] (1,0.5) -- (1,1) -- (0.5,1.5)--(1,1)--(1.125,1.125)--(0.875,1.375)--(1,1.5)--(0.875,1.375)--(0.75,1.5) --(0.875,1.375)--(1.125,1.125)--(1.375,1.375)--(1.25,1.5)--(1.375,1.375) -- (1.5,1.5);
				\draw[thick] (0.75,1.55) arc (30:150:0.15cm);\draw[thick] (1.5,1.55) arc (30:150:0.15cm);
}}}}
\newcommand{\twothreeonetwoloopirreg}{%
	\vcenter{\hbox{\tikz[node distance=2.5ex]{%
				\draw[thick] (7,0.5) -- (7,1) -- (6.5,1.5) -- (6.75,1.25)--(7,1.5)--(6.875,1.375)--(6.75,1.5)-- (6.875,1.375) --(6.75,1.25)-- (7,1)-- (7.5,1.5);
				\draw[thick] (7.5,1.55) arc (45:135:0.36cm);\draw[thick] (7,1.55) arc (30:150:0.15cm);
}}}}
\newcommand{\threeonetwotwoloopirreg}{%
	\vcenter{\hbox{\tikz[node distance=2.5ex]{%
				\draw[thick] (7,0.5) -- (7,1) -- (6.5,1.5) -- (7,1) -- (7.5,1.5)--(7.25,1.25)--(7.125,1.375)--(7.25,1.5)--(7.125,1.375)--(7,1.5);\draw[thick] (7.25,1.55) arc (30:150:0.13cm);
				\draw[thick] (7.5,1.55) arc (30:150:0.13cm);
}}}}
\newcommand{\threetwoonetwoloopirreg}{%
	\vcenter{\hbox{\tikz[node distance=2.5ex]{%
				\draw[thick] (4,0.5) -- (4,1) -- (3.5,1.5) -- (4,1) -- (4.5,1.5)--(4.375,1.375)--(4.25,1.5)--(4.375,1.375)--(4.25,1.25)--(4,1.5);
				\draw[thick] (4.25,1.55) arc (30:150:0.13cm);\draw[thick] (4.5,1.55) arc (30:150:0.13cm);
}}}}
\begin{document}
\title{A quantization of the Loday-Ronco Hopf algebra}
\author{Jo\~{a}o N. Esteves
}
\address{CAMGSD, Departamento de Matem\'{a}tica, Instituto Superior T\'{e}cnico, Av. Rovisco Pais 1, 1049-001 Lisboa, Portugal}
\email{joao.n.esteves@tecnico.ulisboa.pt}
\keywords{Hopf Algebras, Topological Recursion, Matrix Models, Cohomology of Algebras}
\begin{abstract}
We propose a quantization algebra of the Loday-Ronco Hopf algebra $k[Y^\infty]$, based on the Topological Recursion formula of Eynard and Orantin. We have shown in previous works that the Loday-Ronco Hopf algebra of planar binary trees is a space of solutions for the genus 0 version of Topological Recursion, and that an extension of the Loday Ronco Hopf algebra as to include some new graphs with loops is the correct setting to find a solution space for arbitrary genus. Here we show that this new algebra $k[Y^\infty]_h$ is still a Hopf algebra that can be seen in some sense to be made precise in the text as a quantization of the Hopf algebra of planar binary trees, and that the solution space of Topological Recursion $\mathcal{A}^h_{\text{TopRec}}$ is a subalgebra of a quotient algebra $\mathcal{A}_{\text{Reg}}^h$ obtained from $k[Y^\infty]_h$ that nevertheless doesn't inherit the Hopf algebra structure. We end the paper with a discussion on the cohomology of $\mathcal{A}^h_{\text{TopRec}}$ in low degree.
\end{abstract}
\maketitle
\tableofcontents

\section{Introduction}
Topological Recursion is a recursion formula found by B. Eynard and N. Orantin \cite{MR2346575} that has its genesis in matrix models and found applications in various fields, from algebraic geometry \cite{DumitrescuMulase2018,Eynard2015} to integrable systems \cite{eynard2019geometry} and combinatorics \cite{MulaseSulkowski2015}.
For a nice general overview and some references see the ``Preface'' in \cite{liu2018topological}.

In \cite{1751-8121-48-44-445205} we proposed a solution to Topological Recursion in genus $g=0$ based on planar binary trees (pbt) and in $g>0$ based on a new type of graphs with loops built from pbt. The elementary tree $\tree$ and the elementary one loop graph $\oneloop$ built from it by identification of its two leaves are the simplest nontrivial examples of these graphs\footnote{A resum\'{e} of the proprieties of Loday-Ronco's Hopf algebra of pbt \cite{MR1654173} is given in the next section.}. We would like to see these as prototypes of a pair of pants and a one-punctured torus, respectively, that are the building blocks of topological or Riemann surfaces of arbitrary genus with borders or marked points and negative Euler characteristic, even though we won't elaborate much about this analogy on this work. 

Here we will mainly address the issue left open in \cite{MR3888785,1751-8121-48-44-445205} that if the set of the new type of graphs with loops build from pbt still has the structure of a Hopf algebra. We will answer affirmatively to this question with some caveats, showing that the new Hopf algebra $k[Y^\infty]_h$ can be seen as a quantization of Loday-Ronco's Hopf algebra $k[Y^\infty]$ in the following sense: we see the Hopf algebra of pbt as the set of Feynman diagrams of correlation functions at tree level (``classical'' approximation) and the full algebra of loop graphs $k[Y^\infty]_h$ as an expansion in Planck's constant $h$ of some physical theory still to specify\footnote{In this respect see the recent work \cite{wang2021formalism}}. We identify the $h$ expansion with the loop number expansion as is usual in Quantum Field Theory \cite{RevModPhys.47.165}, ignoring possible additional contributions to $h$ powers from external legs of graphs. The algebra $k[Y^\infty]_h$ contains graphs that can not be identified with Riemann surfaces and that we call irregular. By projecting these graphs onto 0 we get a quotient algebra $\mathcal{A}_\text{Reg}^h\subset k[Y^\infty]_h$ that doesn't inherit the structure of a Hopf algebra, because the restriction of the co-product fails to be a homomorphism. By considering words in two generators $\tree$ and $\oneloop$ we get a further restriction $\mathcal{A}_\text{TopRec}^h\subset\mathcal{A}_\text{Reg}^h$ that is the natural space of solutions of the Topological Recursion formula.

The operation of producing a loop between two consecutive leaves in a graph has naturally a differential nature. We show that with a total differential $d_h$ build from this operation, $(k[Y^\infty]_h,d_h)$ is a co-chain complex, graded by the loop number $g$ and the order $n$ of the underlying trees of the graphs $t^g\in k[Y^\infty]_h$ that we identify with the inverse of the Euler characteristic of the associated Riemann surface for regular graphs. We end the paper with a discussion of the cohomology in low degree of the sub-algebra $\mathcal{A}_\text{TopRec}^h$.
\section{Hopf algebra of Loday-Ronco}
 We collect here some important facts of the Loday-Ronco Hopf algebra. Details and proofs can be found in \cite{MR2194965,MR1654173}.
Let $S_n$ be the symmetric group of order $n$ with the usual product $\rho\cdot\sigma$ given by the composition of permutations. When necessary we denote a permutation $\rho$ by its image $(\rho(1)\rho(2)\dots\rho(n))$. Recall that a shuffle $\rho(p,q)$ of type $(p,q)$ in $S_n$ is a permutation such that $\rho(1)<\rho(2)<\dots <\rho(p)$ and $\rho(p+1)<\rho(p+2)<\dots <\rho(p+q)$, with $p+q=n$. For instance the shuffles of type $(1,2)$ in $S_3$ are $(123),(213)$ and $(312)$. We denote the set  of $(p,q)$ shuffles by $S(p,q)$.
Take 
\begin{equation}
k[S^\infty]=\oplus_{n=0}^{\infty}k[S_n]
\end{equation}
with $S_0$ identified with the empty permutation. $k[S^\infty]$ is a vector space over a field $k$ of characteristic $0$ generated by linear combinations of permutations. It is graded by the order of permutations and $k[S_0]$ which is generated by the empty permutation is identified with the field $k$. For two permutations $\rho\in S_p$ and $\sigma\in S_q$ there is a natural product on $S^\infty$ denoted by $\rho\times\sigma$ which is a permutation on $S_{p+q}$ given by letting $\rho$ acting on the first $p$ variables and $\sigma$ acting on the last $q$ variables.

There is a unique decomposition of any permutation $\sigma\in S_n$ in two permutations $\sigma_i\in S_i$ and $\sigma'_{n-i}\in S_{n-i}$ for each $i$ such that
\begin{equation}
\sigma =(\sigma_i\times\sigma'_{n-i})\cdot w^{-1}
\end{equation}
where $w$ is a shuffle of type $(i,n-i)$.
With the $\ast$ product
\begin{equation}
\rho\ast \sigma=\sum_{\alpha_{p,q}\in S(p,q)}\alpha_{p,q}\cdot\left(\rho\times\sigma\right)
\end{equation}
and the co-product
\begin{equation}\label{eq:coproductperm}
\Delta\sigma=\sum \sigma_{i}\otimes\sigma^{'}_{n-i}
\end{equation}
$k[S^{\infty}]$ becomes a bi-algebra and since it is graded and connected it is automatically a Hopf Algebra.

A planar binary tree is a connected graph with no loops embedded in the plane, with only trivalent internal vertices (Fig. \ref{fig:pbt3-2}). In every planar binary tree there are paths that start on a special external edge called the root and end on the other external edges called leaves. It is customary in the subject to ignore terminal vertices that are extremities of order 1 of external edges. The leaves can be left or right oriented. The order $|t|$ of a planar binary tree $t$ is the number of its (internal) vertices and on each planar binary tree of order $n$ there are $n+1$ leaves that usually are numbered from 0 to $n$ from left to right. It is frequent to visualize planar binary trees from the bottom to the top, with the root as its lowest vertical edge and the leaves as the highest edges, oriented SW-NE or SE-NW. We will denote the set of planar binary trees of order $n$ by $Y^n$ and by $k[Y^\infty]$ the graded vector space over $k$ generated by planar binary trees of all orders.
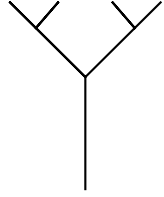
\begin{figure}
	\begin{tikzpicture}
	\draw[thick] (1,-0.5) -- (1,1) -- (0,2) -- (0.35,1.65)--(0.65,2)--(0.35,1.65)-- (1,1) --(1.65,1.65)--(1.35,2)--(1.65,1.65) -- (2,2);
	\end{tikzpicture}
	\caption{A planar binary tree of order 3}\label{fig:pbt3-2}
\end{figure}
Additionally a planar binary tree with levels is a planar binary tree such that on each horizontal line there is at most one vertex.
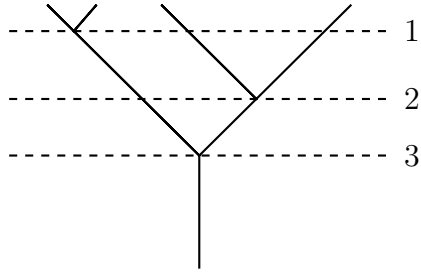
\begin{figure}
	\begin{tikzpicture}
	\draw[thick] (4,-0.5) -- (4,1) -- (2,3) -- (2.35,2.65)--(2.65,3)--(2.35,2.65)-- (4,1) --(4.75,1.75)--(3.5,3)-- (4.75,1.75)--(6,3); \draw[thick,dashed](1.5,1.75)--(6.5,1.75);\draw[thick,dashed](1.5,2.65)--(6.5,2.65);
	\draw (6.8,2.65) node{1}; \draw (6.8,1.75) node{2};\draw[thick,dashed](1.5,1)--(6.5,1);\draw (6.8,1) node{3};
	\end{tikzpicture}
	\caption{Planar binary tree with levels that is the image of $\mathbf{(132)}$}\label{fig:pbtlev3}
\end{figure}
It is clear that reading the vertices from left to right and from top to bottom it is possible to assign a permutation of order $n$ to a planar binary tree with levels and that this assignment is unique. For example in fig. \ref{fig:pbtlev3} the tree corresponds to the permutation $(132)$. In this way it is completely equivalent to consider the Hopf algebra $k[S^\infty]$ or the Hopf algebra of planar binary trees with levels because they are isomorphic. However Loday and Ronco show in \cite{MR1654173} that the $\ast$ product and the co-product are internal on the algebra of planar binary trees which is then isomorphic to a Hopf sub-algebra of $k[S^\infty]$ with the same product and co-product. The identity of the Hopf Algebra $k[Y^\infty]$ is the tree with a single edge and no vertices, following the convention of considering only internal vertices, which represents the empty permutation, and the trivial permutation of $S_1$ is represented by the tree with one vertex and two leaves, see fig \ref{fig:idgen}. In fact this element is the generator of the augmented algebra by the $\ast$ product. See fig. (\ref{fig:1star1star1}) for an example of an order 3 product.

\begin{figure}
	\begin{tikzpicture}
	\draw (0,1) node{ $\Large{1=}$};\draw[thick] (1,-0.25) -- (1,1.75);
	\draw (3.5,1) node{ $\Large{(\mathbf{1})=}$};\draw[thick] (5,-0.5) -- (5,1) -- (4,2) -- (5,1) -- (6,2) ;
	\end{tikzpicture}
	\caption{The identity and the generator in $k[Y^\infty]$}\label{fig:idgen}
\end{figure}
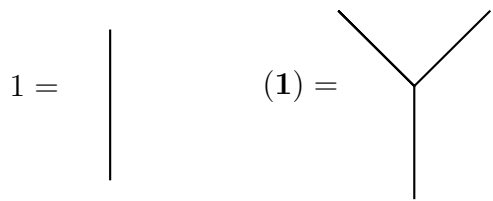
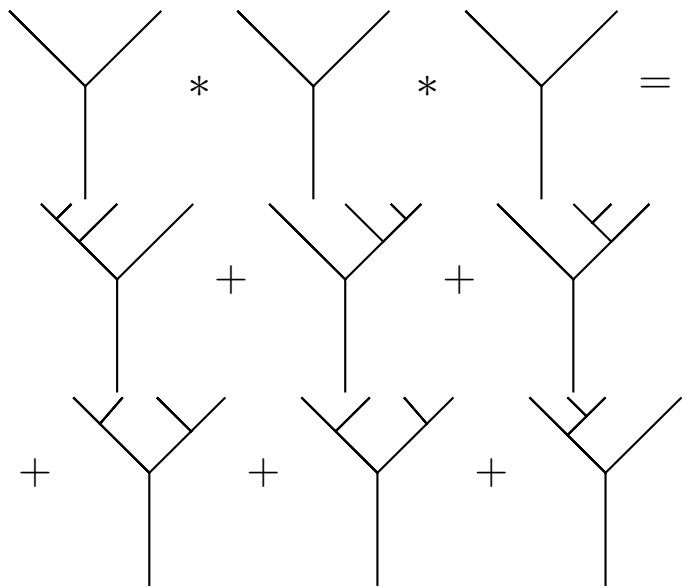
\begin{figure}
	\begin{tikzpicture}
	\draw[thick](1,-0.5) -- (1,1) -- (0,2) -- (1,1) -- (2,2) ; \draw (2.5,1) node{\textbf{{\Large $\ast$}}}; \draw[thick](4,-0.5) -- (4,1) -- (3,2) -- (4,1) -- (5,2); \draw (5.5,1) node{\textbf{{\Large $\ast$}}};\draw[thick](7,-0.5) -- (7,1) -- (6,2) -- (7,1) -- (8,2);\draw (8.5,1) node{\textbf{{\Large $=$}}};
	\end{tikzpicture}
	\begin{tikzpicture}
	\draw[thick] (1,-0.5) -- (1,1) -- (0,2) -- (0.2,1.8)--(0.4,2)--(0.2,1.8)-- (0.5,1.5)--(1,2)--(0.5,1.5)-- (1,1) -- (2,2);\draw (2.5,1) node{\textbf{{\Large $+$}}}; \draw[thick] (4,-0.5) -- (4,1) -- (3,2) -- (4,1) -- (5,2)--(4.8,1.8)--(4.6,2)--(4.8,1.8)--(4.5,1.5)--(4,2);\draw (5.5,1) node{\textbf{{\Large $+$}}}; \draw[thick] (7,-0.5) -- (7,1) -- (6,2) -- (7,1) -- (8,2)--(7.5,1.5)--(7.25,1.75)--(7.5,2)--(7.25,1.75)--(7,2);
	\end{tikzpicture}
	\begin{tikzpicture}
	\draw (-0.5,1) node{\textbf{{\Large $+$}}};\draw[thick] (1,-0.5) -- (1,1) -- (0,2) -- (0.35,1.65)--(0.65,2)--(0.35,1.65)-- (1,1) --(1.55,1.55)--(1.10,2)--(1.55,1.55) -- (2,2);\draw (2.5,1) node{\textbf{{\Large $+$}}};  \draw[thick] (4,-0.5) -- (4,1) -- (3,2) -- (3.45,1.55)--(3.9,2)--(3.45,1.55)-- (4,1) --(4.65,1.65)--(4.35,2)--(4.65,1.65) -- (5,2);\draw (5.5,1) node{\textbf{{\Large $+$}}}; \draw[thick] (7,-0.5) -- (7,1) -- (6,2) -- (6.5,1.5)--(6.75,1.75)--(6.5,2)--(6.75,1.75)--(7,2)-- (6.5,1.5) -- (7,1)-- (8,2);
	\end{tikzpicture} 
	\caption{$\mathbf{(1)}\ast\mathbf{(1)}\ast\mathbf{(1)}=\mathbf{(123)}+\mathbf{(321)}+\mathbf{(312)}+\mathbf{(132)}+\mathbf{(231)}+\mathbf{(213)}$ computed in $k[S^\infty]$. Note that in $k[Y^\infty]$ the fourth and the fifth trees are the same.} \label{fig:1star1star1}
\end{figure}
The grafting $t_1 \vee t_2$ of two trees $t_1$ and $t_2$ is the operation of producing a new tree $t$ by inserting $t_1$ on the left and $t_2$ on the right leaves of $\mathbf{(1)}$. It is clear that any tree of order $n$ can be written as $t_1\vee t_2$ with $t_1$ of order $p$, $t_2$ of order $q$ and $n=p+q+1$. If a tree has only leaves on the right branch besides the first leaf then it can be uniquely written as $|\vee t_2$ and reciprocally if it has only leaves on the left branch besides the last leaf. Note that in particular $\mathbf{(1)}=|\vee |$. 
In \cite{MR1654173} Loday and Ronco show that the $\ast$ product restricted to planar binary trees satisfies the identity
\begin{equation}\label{eq:shuffleident1}
t\ast t'=t_1\vee (t_2\ast t')+(t \ast t_1^{'})\vee t_2^{'}
\end{equation}
and
\begin{equation}\label{eq:shuffleident2}
t\ast |=|\ast t=t
\end{equation}
with $t=t_1\vee t_2$ and $t'=t_1^{'}\vee t_2^{'}$.
\section{The relation with Eynard-Orantin topological recursion formula}\label{sec:toprec}
The topological recursion formula of Eynard and Orantin has its origin in Matrix Models, for general reviews see for instance \cite{DiFrancesco:1993nw,MR2346575}. In the hermitian 1-matrix form of the theory the purpose is to compute connected correlation functions $W_{k+1}$ depending on a set of variables $p, p_1,\dots,p_k$
\begin{equation}
W_{k+1}(p,p_1,\dots p_k)=\Braket{\text{Tr}\frac{1}{p-M}\text{Tr}\frac{1}{p_1-M}\dots \text{Tr}\frac{1}{p_k-M}}_c
\end{equation}
starting with $W_1(p)$ and $W_2(p,p_1)$.
These functions which are solutions of the so-called loop equations are only well defined over Riemann surfaces because in $\Complex$ they are multi-valued. They admit an expansion on the order $N$ of the random matrix $M$, with components $W_{k+1}^g(p,p_1,\dots p_k)$ related to a definite genus. We will not be concerned here with the actual computation of correlation functions in specific models, except for the example below.

Let $K=(p_1,\dots,p_k)$ be a vector of variables. For instance in concrete cases these can be coordinates of punctures on Riemann surfaces, labels of borders on topological surfaces or variables in Matrix Models, but we just leave them as labels of leaves of planar binary trees or of graphs obtained from planar binary trees. We assign the label $p$ to the root of a tree or of a graph with loops obtained from a tree. The topological recursion formula is
\begin{align}\label{eq:toprec}
&W_{k+1}^g(p,K)=\sum_{\text{branch points }\alpha}\text{Res}_{q\rightarrow \alpha}K_p(q,\bar{q})\notag\\
&\left(W^{g-1}_{k+2}(q,\bar{q},K)+\sum_{L\cup M=K,i=0}^g W^i_{|L|+1}(q,L)W^{g-i}_{|M|+1}(\bar{q},M)\right)
\end{align}
where the sum is restricted to terms with Euler characteristic equal or smaller than 0. For instance if $i=0$ then $|L|\ge 1$. 
For a very clear exposition about this setup from the point of view of Algebraic Geometry see for instance \cite{MR3087960} but some comments are in order. The branch points are the ones from a meromorphic function $x$ defined on a so called spectral curve $\mathcal{E}(x,y)=0$. The recursion kernel $K_p(q,\bar{q})$ is, roughly speaking, a meromorphic (1,1) tensor that depends on a regular point $p$ and on $q$ in the neighborhood of a branch point and its Galois conjugated point $\bar{q}$ for which $x(q)=x(\bar{q})$ and $y(q)=-y(\bar{q})$\footnote{Due to a general phenomena known as Symplectic Invariance (see \cite{MR2346575}) the theory is invariant under the transformation $x\rightarrow y, y\rightarrow x$ so the previous conditions can also be expressed under this permutation. }. In fact the recursion kernel can be computed from the spectral curve and $W_2^0(p,p_1)$ which is a symmetric meromorphic differential of order two. (see the example below). Actually, all $W_k^g$ are meromorphic symmetric differentials but we will continue to refer to them as correlation functions. Since our approach will be purely algebraic and in order to soften the notation we will not explicitly mention the sum of the residues over the branch points when referring to this formula. The only exception is the following
\begin{defn2}[The Airy spectral curve]\label{ex:Airy}
The Airy function 
\begin{equation}
	Ai(x)=\frac{1}{2\pi}\int_{-\infty}^{\infty}e^{ipx}e^{i\frac{p^3}{3}}dp,
\end{equation}	
for $x\in\Real$, is the solution of the ordinary differential equation
\begin{equation}
	\left(\frac{d^2 }{dx^2}-x\right)Ai(x)=0.
\end{equation}
It can be seen as the Schr\"{o}dinger equation that results from the quantization of the algebraic curve
\begin{equation}\label{eq:Airy}
	y^2-x=0
\end{equation}
 under the quantization rule that replaces the variable $x$ by the operator that multiplies a function in a convenient Hilbert space by the independent variable $x$ and the variable $y$ by the ``momentum'' operator $-id/d x$:
 \begin{equation}
 	\left(\left(-i\frac{d}{d x}\right)^2+x\right)Ai(x)=0.
 \end{equation}

Following \cite{MR2346575}, we will consider (\ref{eq:Airy}) with $x,y\in\Complex$ as the spectral curve for the Topological Recursion formula (\ref{eq:toprec}) and show how compute the correlation functions from it. 

We choose the following parametrization of the Airy curve:
\begin{equation}
	x(z)=z^2, y(z)=z, z\in\Complex.
\end{equation}
The branching points $\alpha$ are the solutions of $dx(z)=0$ which in this case is just $\alpha: z=0$. Then the Galois conjugate points are $y(q)=q$ and $y(\overline{q})=\overline{q}=-q$ for $q$ in a neighborhood of 0.
Since (\ref{eq:Airy}) is a curve of genus $0$ the so called Bergman Kernel $B(p,q)$ which is a symmetrical meromorphic 2-tensor with no residues and a double pole at $p=q$ is particularly simple:
\begin{equation}
	B(p,q)= \frac{1}{(p-q)^2}dp\otimes dq.
\end{equation}
Then the recursion kernel $K_p(q,\overline{q})$ is given by
\begin{equation}
	K_p(q,\overline{q})=\frac{1}{\omega(q)}\int_q^{\overline{q}}B(\cdot,p)
\end{equation}
where $\omega(q)=(y(q)-y(\overline{q}))dx(q)$.
Explicitly,
\begin{equation}
	K_p(q,\overline{q})=\frac{1}{4q(q^2-p^2)}dp\otimes\frac{1}{dq}
\end{equation}
where $\overline{q}=-q$ has already been assumed. The $dq$ factor in the denominator is a remind that $K_p(q,\overline{q})$ is a mixed tensor that annihilates a $dq$ one-form. Also $W^0_2(p,p_1)$ in this case is just the Bergman kernel: 
\begin{equation}
	W^0_2(p,p_1)=\frac{1}{(p-p_1)^2}dp\otimes dp_1
\end{equation}
From (\ref{eq:toprec}) we can get the expression of $W^0_3(p,p_1,p_2)$:
\begin{align}\label{eq:toprecW03}
	W^0_3(p,p_1,p_2)&=\text{Res}_{q\rightarrow 0}K_p(q,\bar{q})\\\notag
	&\left(W^0_2(p_1,q)W^0_2(p_2,\overline{q})+W^0_2(p_2,q)W^0_2(p_1,\overline{q})\right).
\end{align}
The computation of the residue offers no difficulty if we remember that the integration measure $dq$ is contained in the argument of $\text{Res}_{q\rightarrow 0}$ in (\ref{eq:toprecW03}) and the result is
\begin{equation}
	W^0_3(p,p_1,p_2)=\frac{1}{2p^2p_1^2p_2^2}dp\otimes dp_1\otimes dp_2.
\end{equation}

Similarly for $W^1_1(p)$ we have
\begin{align}
	W^1_1(p)=&\text{Res}_{q\rightarrow 0}K_p(q,\bar{q})W^0_2(q,\overline{q})\\\notag
	&=\left(\text{Res}_{q\rightarrow 0}\frac{1}{4q(q^2-p^2)}\frac{d\overline{q}}{(q-\overline{q})^2}\right)\otimes dp\\\notag
	&=-\frac{1}{16}\left(\text{Res}_{q\rightarrow 0}\frac{1}{q^3(q-p)(q+p)}dq\right)\otimes dp\\\notag
	&=\frac{1}{16}\left(\text{Res}_{q\rightarrow \pm p}\frac{1}{q^3(q-p)(q+p)}dq\right)\otimes dp\\\notag
	&=\frac{dp}{(2p)^4}.
\end{align}
\end{defn2}

\subsection{Genus 0.}\label{sec:toprecg0}
 The existence of a sort of representation map $\psi$ from the vector space of correlation functions of genus $g$ to the vector space of graphs with loops was suggested in \cite{1751-8121-48-44-445205} such that in particular to a correlation function $W^0_{n+2}(p,p_1,\dots,$$p_{n+1})$ of Euler characteristic $\chi=-n$ and genus 0 would correspond the trees of order $n$. In fact the way that this map was defined implied that the representation of $W^0_{n+2}(p,p_1,\dots,p_{n+1})$ is the sum of all trees of order $n$. In \cite{MR3888785} we saw that the map $\psi$ allows to give a ring structure to the space of correlation functions that obey the Eynard-Orantin formula, with an identity that is given by the cylinder, $W^0_2(p,p_1)$. In fact it is a consequence of the axioms of topological quantum field theory as stated by Atiyah for example in \cite{atiyah1988topological} that the cylinder $\Sigma\times I$, where $\Sigma$ is a topological surface without border and $I$ is a interval of real numbers, may be identified with the identity map in a vector space. 
 
 \begin{defn}\label{def:W3b}\cite{1751-8121-48-44-445205}
 	The propagator or cylinder (also named Bergman kernel in the literature) $W^0_2(q,\bar{q})$ is represented through $\psi$ by the empty permutation $|$ and the recursion kernel is represented through $\psi$ by $\tree$ when in an internal vertex of some tree. Then each planar binary tree of order n is a representation of an instance of some correlation function in genus 0 with each vertex identified with a recursion kernel and each left leaf identified with the cylinder $W^0_2(q_i,p_j)$ or each right leaf identified with the cylinder $W^0_2(\bar{q}_i,p_k)$. Finally the image under $\psi$ of a correlation function $W_{n+2}^0(p,p_1,\dots,p_{n+1})$ with $\chi=-n$ is the sum of all planar binary trees of order $n$ considering all permutations of their leaf labels and with the identifications mentioned above,
 	\begin{equation}
 	\psi\left(W_{n+2}^0(p,p_1,\dots,p_{n+1})\right)=\sum_{\substack{t_i\in Y^n \\ \text{perm. of leaf labels $\{p_1,\dots,p_{n+1}\}$}}} t_i.
 	\end{equation}
 \end{defn}

\begin{defn2}\label{ex:W3}
	Consider the planar binary tree with one vertex. The 3-point correlation function $W_3^0(p,p_1,p_2)$ is represented by the sum of two planar binary trees with one vertex, obtained by the permutation of the leaf labels $p_1$ and $p_2$.
\end{defn2}
\begin{align}
\psi\left(W_3^0(p,p_1,p_2)\right)&=\psi\left(K_p(q,\bar{q})W^0_2(q,p_1)W^0_2(\bar{q},p_2)\right)+ \text{ perm. of $\{p_1,p_2\}$}\notag\\
&=\sum_{\text{ perm. of $\{p_1,p_2\}$}}|\vee |\notag\\
&=\sum_{\text{ perm. of $\{p_1,p_2\}$}}(\mathbf{1})
\end{align}

\begin{prop}\cite{1751-8121-48-44-445205}
	If $W_{n+2}^0(p,p_1,\dots,p_{n+1})$ is a correlation function with Euler characteristic $\chi=-n$ that is a solution of (\ref{eq:toprec}) then we have
	\begin{align}\label{eq:defcorr}
	\psi \left(W_{n+2}^0(p,p_1,\dots,p_{n+1})\right)&=\sum_{\substack{p+q+1=n\\|t_1|=p, |t_2|=q}} t_1\vee t_2\notag\\
	& + \text{perm. of leaf labels $\{p_1,\dots,p_{n+1}\}$}
	\end{align}
\end{prop}

\begin{thm}\label{thm:corrg0}\cite{1751-8121-48-44-445205}
	The $n$-order solution $W_{n+2}^0(p_1,\dots,p_{n+1})$ of the topological recursion in genus 0 is represented by the linear combination
	$$\sum_{t\in Y^n} t=\mathbf{(1)}\ast\mathbf{(1)}\ast\dots\ast\mathbf{(1)}$$
	with $n$ factors of $\mathbf{(1)}$ followed by the sum over all permutations of its labels. 
\end{thm}
In this way $W_{n+2}^0(p,p_1,\dots,p_{n+1})$ is represented by $\sum_{t\in Y^n} t$ followed by the identification of cylinders $W^0_2(q_i,p_j)$ with the left leaves or $W^0_2(\bar{q}_i,p_k)$ with the right leaves and finally by summing over all permutations of the labels $p_1,p_2,\dots, p_{n+1}$. In other words, the $\ast$ product $\mathbf{(1)}\ast\mathbf{(1)}\ast\dots\ast\mathbf{(1)}$ gives all possible insertions of recursion kernels of $W_{n+2}^0$.

\subsection{Genus higher than 0}\label{sec:toprecgbt0}
The procedure of attaching an edge to two consecutive leaves and producing a graph with loops allows to represent correlation functions with genus $g>0$. This is equivalent to extracting the outermost cylinders $W^0_2(x,p_j),W_2^0(y,p_{j+1}), x=q_i$ or $\bar{q}_i$, $y=\bar{q}_j$ or $q_j$ and to couple a cylinder $W^0_2(x,y)$  to two recursion kernels $K_{q_l}(q_i,\bar{q}_i)$ and $K_{q_m}(q_j,\bar{q}_j)$, for some convenient choice of indices, that are identified with two internal vertices. This procedure does not change the Euler characteristic of the associated correlation functions because the number of pairs of leaf labels is reduced exactly as the genus is increased. For instance with this procedure we can make the sequence
\begin{equation}
W_5^0(p,p_1,\dots p_4)\longrightarrow W^1_3(p,p_1,p_2)\longrightarrow W^2_1(p)
\end{equation} 
and remain in the same graded vector space that contains $k[Y^3]$. How this changes the Hopf algebra structure will be discussed below. For now, we define the operation $_{i}\leftrightarrow_{i+1}$ on a planar binary tree. 
\begin{defn}\label{defn:connecting}\cite{1751-8121-48-44-445205}
	Starting with a planar binary tree of order $n$ and $n+2$ labels (including the root label $p$) the operation $_{i}\leftrightarrow_{i+1}$ consists in erasing the labels of the leaves $i$ and $i+1$ then connecting them by an edge and finally relabeling the remaining leaves, now numbered $j$ with $j=0,\dots,n-2$, with the $p_{j+1}$ labels, producing in this way a graph with one loop.
\end{defn}
Therefore we represent a correlation function $W^g_{k}(p,p_1,\dots,p_{k-1})$ of genus $g$ by graphs with loops $t^g$ that are obtained by successive applications of the $_{i}\leftrightarrow_{i+1}$ operation. We denote by $\left(Y^n\right)^g$ the set of different graphs with $g$ loops that are obtained from trees $t\in Y^n$. We set the total order $\|t^g\|$ of a graphs $t^g\in (Y^n)^g$ to $\|t^g\|=n+g.$
\begin{defn}\cite{1751-8121-48-44-445205}
	A correlation function $W^g_{k}(p,p_1,\dots,p_{k-1})$ of genus $g$ and Euler characteristic $\chi=2-2g-k$ is represented by a sum of all different graphs with loops $t^g\in \left(Y^n\right)^g$ for  $n=-\chi$:
	\begin{equation}
	\psi\bigl(W^g_{k}(p,p_1,\dots,p_{k-1})\bigr)=\sum_{t^g\in (Y^n)^g} t^g
	\end{equation}
\end{defn}
\section{A quantization of Loday-Ronco Hopf algebra }
We start with the module over a ring $k$ of characteristic 0 generated by correlation functions of genus 0 that are solutions of Topological Recursion and which we denote by $\textbf{Corr}$.  We consider it as the direct sum of modules of correlation functions of specific Euler characteristic $\chi=-n$, which we denote by $\textbf{Corr}^{\mathbf{(n)}}$:
\begin{equation}
\textbf{Corr}=\overset{\infty}{\underset{n=0}{\oplus}}\textbf{Corr}^{\mathbf{(n)}}
\end{equation}
so that $\textbf{Corr}$ is a graded module. We show in \cite{MR3888785} that $\textbf{Corr}$ carries an algebra structure with the product between correlation functions of genus 0 given simply by $W^0_{k}\ast W^0_{l}=W^0_{k+l}$. This is an immediate consequence of the fact that a $g=0$ correlation function $W^0_{k}$ associated with an Euler characteristic $\chi=2-k$ is represented by the sum of all pbt of order $-\chi$ that in turn is given by the $\ast$ product of $-\chi$ factors of $\tree$, see Theorem \ref{thm:corrg0}. 

We see $\textbf{Corr}$ as a classical space and consider its quantization $\textbf{Corr}_{\mathbf{h}}$ to be the space of correlation functions with $g>0$. In the same way we have
\begin{equation}
\textbf{Corr}_{\mathbf{h}}=\overset{\infty}{\underset{n=0}{\oplus}}\textbf{Corr}^{{\mathbf{(n)}}}_{\mathbf{h}}
\end{equation}
$\textbf{Corr}_{\mathbf{h}}^{\mathbf{(n)}}$ is by definition the free module generated by correlation functions $W^g_k(p,p_1,\dots,p_{k-1})$ with a fixed Euler characteristic $\chi=-n$ and varying genus $g>0$, up to the consistency of the formula $\chi=2-2g-k$, over the ring $k[[h]]$ of formal power series in a parameter $h$ with coefficients in $k$. In this way every correlation function $W^{(n)}\in\textbf{Corr}_{\mathbf{h}}^{\mathbf{(n)}}$ admits an expansion in terms of $h$. We are assuming here the well established idea in Quantum Field Theory that an expansion in the loop number is the same thing as an expansion in $h$ \cite{RevModPhys.47.165}. Also we are ignoring possible contributions of the external edges to powers in $h$. For $\chi=-n$ even,
\begin{align}\label{eq:fullcorrfun}
W^{(n)}=&W^0_{n+2}(p,p_1,\dots,p_{n+1})+hW^1_{n}(p,p_1,\dots,p_{n-1})+\notag\\
&\dots +h^{n/2}W^{n/2}_2(p,p_1).
\end{align}
and for $\chi=-n$ odd,
\begin{align}
W^{(n)}=&W^0_{n+2}(p,p_1,\dots,p_{n+1})+hW^1_{n}(p,p_1,\dots,p_{n-1})+\notag\\
&\dots +h^{(n+1)/2}W^{(n+1)/2}_1(p).
\end{align}
Similarly for each tree $t\in Y^n$ of order $n$ there will be a full quantum graph $t_h$ expanded in powers of $h$ 
\begin{equation}
t_h=t+ht^1+h^2t^2+\dots+h^gt^g
\end{equation}
with $t^1\in \left(Y^n\right)^1$ a graph of order $n$ with one loop, etc... And with terms of order up to $g$ compatible with the relation $-n=2-2g-k$, $k=1$ or 2 and without any or only one free leaf respectively. The tree $t$ is thus the 0 order term of $t_h$. We consider two gradings: the Euler characteristic degree which coincides with minus the order $n$ of the underlying pbt of the graph and the genus or loop degree $g$ which coincides with the power of $h$ in the expansion of the full quantum graph. For the elementary tree,
\begin{equation}
\tree_h =\tree+ h\oneloop
\end{equation}
Then is clear that we have a two term direct sum in the full module of graphs with loops
\begin{equation}
k[Y^\infty]_h=k\left[\oplus_{n=0}^\infty\oplus_{g,2g+k-2=n}\left(Y^n\right)^g\right].
\end{equation}

The product $\ast_h$ in $\textbf{Corr}_{\mathbf{h}}$ that we will discuss below endows it with a ring structure and is a direct consequence of topological recursion formula. At the level of graphs that represent correlation functions of arbitrary genus it is a generalization
of the relation (\ref{eq:shuffleident1}) satisfied by pbt (see \cite{MR1654173}):
\begin{equation}
t_1\ast t_{2}= 
\left(t_{1}\ast t_{2}^{'}\right)\vee t^{''}_{2}+t^{'}_{1}\vee \left(t^{''}_{1}\ast t_{2}\right)
\end{equation}
for $t_1=t_1^{'}\vee t_1^{''}$ and $t_2=t_2^{'}\vee t_2^{''}$ two planar binary trees of arbitrary order decomposed on their left and right branches. We use this relation as the base of the generalization to graphs with loops \footnote{Remember the operator $\bridge$, defined in \cite{1751-8121-48-44-445205} that, for $t^{g_1}\in(Y^{n_1})^{g_1},t^{g_2}\in(Y^{n_2})^{g_2}$, produces a new graph $t^{g}=t^{g_1}\bridge t^{g_2}\in(Y^{n})^g$, with $n=n_1+n_2+1, g=g_1+g_2+1$, by making $t^{g_1}$ and $t^{g_2}$ the left and right branches of $t^{g}$, respectively, and producing a new loop by identifying with an additional edge the last leaf of $t^{g_1}$ with the first leaf of $t^{g_2}$, whenever this is possible.}
\begin{defn}
	\begin{align}\label{eq:quantumproduct}
	t^{g_1}\ast_h t^{g_2}= 
	\begin{cases}
	&\left(t^{g_1}\ast_h t^{g^{'}_2}\right)\vee t^{g^{''}_{2}}+t^{g^{'}_{1}}\vee \left(t^{g^{''}_{1}}\ast_h t^{g_{2}}\right),\\
	&\text{if } t^{g_{1}}=t^{g^{'}_{1}}\vee t^{g^{''}_{1}}\text{ and } t^{g_{2}}=t^{g^{'}_{2}}\vee t^{g^{''}_{2}};\\
	&\left(t^{g_1}\ast_h t^{g^{'}_2}\right)\vee t^{g^{''}_{2}}+t^{g^{'}_{1}}\bridge \left(t^{g^{''}_{1}}\ast_h t^{g_{2}}\right), \\
	&\text{if } t^{g_{1}}=t^{g^{'}_{1}}\bridge t^{g^{''}_{1}}\text{ and } t^{g_{2}}=t^{g^{'}_{2}}\vee t^{g^{''}_{2}};\\
	&\left(t^{g_1}\ast_h t^{g^{'}_2}\right)\bridge t^{g^{''}_{2}}+t^{g^{'}_{1}}\vee \left(t^{g^{''}_{1}}\ast_h t^{g_{2}}\right), \\
	&\text{if } t^{g_{1}}=t^{g^{'}_{1}}\vee t^{g^{''}_{1}}\text{ and } t^{g_{2}}=t^{g^{'}_{2}}\bridge t^{g^{''}_{2}};\\
	&\left(t^{g_1}\ast_h t^{g^{'}_2}\right)\bridge t^{g^{''}_{2}}+t^{g^{'}_{1}}\bridge \left(t^{g^{''}_{1}}\ast_h t^{g_{2}}\right), \\
	&\text{if } t^{g_{1}}=t^{g^{'}_{1}}\bridge t^{g^{''}_{1}}\text{ and } t^{g_{2}}=t^{g^{'}_{2}}\bridge t^{g^{''}_{2}},
	\end{cases}
	\end{align}
	for $t^{g_1}\in \left(Y^{n_1}\right)^{g_1},t^{g_2}\in \left(Y^{n_2}\right)^{g_2}$,where $g_1=g^{'}_1+g^{''}_1$ and $g_2=g^{'}_2+g^{''}_2$ in the first case, $g_1=g^{'}_1+g^{''}_1+1$ and $g_2=g^{'}_2+g^{''}_2$ in the second case, $g_1=g^{'}_1+g^{''}_1$ and $g_2=g^{'}_2+g^{''}_2+1$ in the third case and $g_1=g^{'}_1+g^{''}_1+1$ and $g_2=g^{'}_2+g^{''}_2+1$ in the fourth case.
\end{defn}
We need to compute $\oneloop\ast_h\oneloop$. The previous formula forces to consider graphs with two independent loops contracted to the same leaf. In fact,
\begin{align}
	\oneloop\ast_h\oneloop&=(|\bridge|)\ast_h(|\bridge|)\notag\\
	&=|\bridge\oneloop+\oneloop\bridge|\notag\\
	&=\twoonelooponeirreg+\onetwolooponeirreg.
\end{align}
These type of graphs don't satisfy the Riemann surface's Euler characteristic formula $\chi=2-2g-k$ since the order of the underlying tree is 2 which should be equal to $-\chi$, the genus is $g=2$ and the number of punctures is $k=1$. So we can't associate a Riemann surface to this kind of graphs as we do with $\onetwo$ or $\onetwoloopone$ or any other graph that has only one loop contracted with each leaf. We call these graphs irregular:
\begin{defn}
	An irregular graph is a graph $t^g\in (Y^n)^g$ that has two edges associated to two loops contracted in the same leaf of the underlying pbt.
\end{defn} 

In the following we give some examples of computation of products of graphs.
 
\begin{defn2}
	\begin{align}
	\twooneloopone\ast_h\oneloop&=(|\bridge\tree)\ast_h(|\bridge|)\notag\\
	&=(\twooneloopone\ast_h|)\bridge |+|\bridge(\tree\ast_h\oneloop)\notag\\
	&=\twooneloopone\bridge|+|\bridge\left(\twoonelooptwo+\onetwolooptwo\right)\notag\\
	&=\threetwolooptwo+\threetwoloopthree+\threetwoloopone.
	\end{align}
\end{defn2}
 \begin{defn2}
 	\begin{align}
 	\twooneloopone\ast_h\twoonelooptwo=&(\twooneloopone\ast_h |)\vee\oneloop+|\bridge(\tree\ast_h\twoonelooptwo)\notag\\
 	=&\twooneloopone\vee\oneloop+|\bridge\left(\threeoneloopone+\threeonelooptwo\right)\notag\\
 	=&\fourtwooloopone+\fourtwooloopthree+\fourtwooloopfour.
 	\end{align}
 \end{defn2}
\begin{defn2}
	\begin{align}
		\twoonelooptwo\ast_h\oneloop&=(|\vee\oneloop)\ast_h(|\bridge|)\notag\\
		&=|\vee(\oneloop\ast_h\oneloop)+\twoonelooptwo\bridge |\notag\\
		&=|\vee(\twoonelooponeirreg+\onetwolooponeirreg)+\twothreeonetwoloopirreg\notag\\
		&=\threetwoonetwoloopirreg+\threeonetwotwoloopirreg+\twothreeonetwoloopirreg\
	\end{align}
\end{defn2}
\begin{thm}
	The module $k[Y^\infty]_h$ with the product (\ref{eq:quantumproduct}) is an associative algebra.
\end{thm}
\begin{proof}
	We must show that the product (\ref{eq:quantumproduct}) is associative:
	\begin{equation}\label{eq:associative}
	t^{g_{1}}\ast_h\left(t^{g_{2}}\ast_h t^{g_{3}}\right)=\left(t^{g_1}\ast_h t^{g_{2}}\right)\ast_h t^{g_{3}}
	\end{equation}
	The expression (\ref{eq:quantumproduct}) for the product distinguishes the cases where the graphs $t^g$ decomposes as $t^g=t^{g'}\vee t^{g''}$ or $t^g=t^{g'}\bridge t^{g''}$. Then in (\ref{eq:associative}) we must consider the 6 possibilities for $t^{g_1},t^{g_2}$ and $t^{g_3}$. We will focus on the case where $t^{g_1}=t^{g_{1}^{'}}\bridge t^{g_{1}^{''}},t^{g_{2}}=t^{g_{2}^{'}}\vee t^{g_{2}^{''}}$ and $t^{g_{3}}=t^{g_{3}^{'}}\vee t^{g_{3}^{''}}$, the other cases being similar. We proceed by induction on the genus, the case of genus 0 being the original Loday-Ronco product which is known to be associative. For the total genus $g=g_1+g_2+g_3$ and after some algebra, the left-hand side of (\ref{eq:associative}) is equal to
	\begin{align}
	&t^{g^{'}_1}\bridge\Biggl(t^{g^{''}_{1}}\ast_h\Bigl(t^{g^{'}_2}\vee\bigl(t^{g^{''}_{2}}\ast_h(t^{g^{'}_{3}}\vee t^{g^{''}_{3}})\bigr)\Bigr)\Biggr)\\
	+&\bigl((t^{g^{'}_{1}}\bridge t^{g^{''}_{1}})\ast_h t^{g^{'}_{2}}\bigr)\vee\bigl(t^{g^{''}_{2}}\ast_h(t^{g^{'}_{3}}\vee t^{g^{''}_{3}})\bigr)\notag\\
	+&t^{g^{'}_{1}}\bridge\Bigl(t^{g^{''}_{1}}\ast_h\bigl((t^{g^{'}_{2}}\vee t^{g^{''}_{2}})\ast_ht^{g^{'}_{3}}\bigr)\vee t^{g^{''}_{3}}\Bigr)\notag\\
	+&\Bigl(\bigl((t^{g^{'}_{1}}\bridge t^{g^{''}_{1}})\ast_h(t^{g^{'}_{2}}\vee t^{g^{''}_{2}})\bigr)\ast_ht^{g^{'}_{3}}\Bigr)\vee t^{g^{''}_{3}},\notag
	\end{align}
	where the associative for a lower total genus was assumed, and the right-hand side is equal to
	\begin{align}
	&t^{g^{'}_{1}}\bridge\Bigl(t^{g^{''}_{1}}\ast_h\bigl((t^{g^{'}_{2}}\vee t^{g^{''}_{2}})\ast_h(t^{g^{'}_{3}}\vee t^{g^{''}_{3}})\bigr)\Bigr)\\
	+&\Biggl(\Bigl(t^{g^{'}_{1}}\bridge\bigl(t^{g^{''}_{1}}\ast_h(t^{g^{'}_{2}}\vee t^{g^{''}_{2}})\bigr)\Bigr)\ast_ht^{g^{'}_{3}}\Biggr)\vee t^{g^{''}_{3}}\notag\\
	+&\bigl((t^{g^{'}_{1}}\bridge t^{g^{''}_{1}})\ast_h t^{g^{'}_{2}}\bigr)\vee\bigl(t^{g^{''}_{2}}\ast_h(t^{g^{'}_{3}}\vee t^{g^{''}_{3}})\bigr)\notag\\
	+&\Bigl(\bigl((t^{g^{'}_{1}}\bridge t^{g^{''}_{1}})\ast_h(t^{g^{'}_{2}}\vee t^{g^{'}_{2}})\bigr)\ast_h t^{g^{'}_{3}}\Bigr)\vee t^{g^{''}_{3}}.\notag
	\end{align}
	Using (\ref{eq:quantumproduct}) both sides are seen to be equal.
\end{proof}

Accordingly the co-product must also be generalized. In \cite{MR1654173} it is shown that the co-product of  a planar binary tree $t=t_1\vee t_2$ of order $n$ satisfies the recursion relation
\begin{equation}
\Delta(t)=\sum_{j,k}t_{1j}^{'}\ast t_{2k}^{'}\otimes t^{''}_{1(p-j)}\vee t^{''}_{2(q-k)}+t\otimes |,
\end{equation}
with the notation convention that $\Delta(t_1)=\sum_j t_{1j}^{'}\otimes t^{''}_{1(p-j)}$ and
$\Delta(t_2)=\sum_j t_{2k}^{'}\otimes t^{''}_{2(q-k)}$ and $n=p+q+1$. We base our definition on this relation and generalize it to the following: 
\begin{defn}\label{def:co-product}
For $t^g\in(Y^{n})^g$, $t^g=t^{g_1}\vee t^{g_2}$, $t^{g_1}\in(Y^p)^{g_1}$,  $t^{g_2}\in(Y^q)^{g_2}$ with $n=p+q+1$, $g=g_1+g_2$,  and $\Delta_h (t^{g_1})=\sum_{j,g^{'}_1,g^{''}_1} (t^{'}_j)^{g'_1}\otimes (t^{''}_{p-j})^{g^{''}_1}, \Delta_h (t^{g_2})=\sum_{k,g^{'}_2,g^{''}_2} (t^{'}_k)^{g'_2}\otimes (t^{''}_{q-k})^{g^{''}_2}$ with $g_1=g^{'}_1+g^{''}_1, g_2=g^{'}_2+g^{''}_2 $ we have
\begin{equation}\label{eq:coprodquantum1}
\Delta_h (t^g)=\sum_{j,k,g^{'}_1,g^{'}_2,g^{''}_1,g^{''}_2}\left((t^{'}_j)^{g^{'}_1}\ast_h(t^{'}_k)^{g^{'}_2}\right)\otimes\left((t^{''}_{p-j})^{g^{''}_1}\vee (t^{''}_{q-k})^{g^{''}_2}\right)+t^g\otimes 1;
\end{equation}
for $t^g=t^{g_1}\bridge t^{g_2}$ with $g=g_1+g_2+1$ and otherwise the same conventions we have
\begin{equation}\label{eq:coprodquantum2}
\Delta_h (t^g)=\sum_{j,k,g^{'}_1,g^{'}_2,g^{''}_1,g^{''}_2}\left((t^{'}_j)^{g^{'}_1}\ast_h(t^{'}_k)^{g^{'}_2}\right)\otimes\left((t^{''}_{p-j})^{g^{''}_1}\bridge (t^{''}_{q-k})^{g^{''}_2}\right)+t^g\otimes 1.
\end{equation}
\end{defn}

An immediate consequence of this definition is that $\oneloop$ is also primitive. In fact, we have
\begin{equation}
\oneloop=|\bridge |
\end{equation}
hence, by (\ref{eq:coprodquantum2}) it follows that
\begin{equation}\label{eq:oneloopprim}
\Delta_h\oneloop=1\otimes \oneloop + \oneloop\otimes 1.
\end{equation}
Other examples are as follows:
\begin{defn2}
\begin{equation}
	\Delta_h\left(\twooneloopone\right)=1\otimes\twooneloopone+\twooneloopone\otimes+\tree\otimes\oneloop
\end{equation}
which is a consequence of (\ref{eq:coprodquantum2}) and (\ref{eq:oneloopprim}).
\end{defn2}
\begin{defn2}
	\begin{align}
		\Delta_h\left(\twoonelooponeirreg\right)&=\sum1^{'}\ast_h\oneloop^{'}\otimes1^{''}\bridge\oneloop^{''}+\twoonelooponeirreg\otimes1\notag\\
		&=1\otimes\twoonelooponeirreg+\twoonelooponeirreg\otimes1+\oneloop\otimes\oneloop
	\end{align}
	which is again a consequence of (\ref{eq:coprodquantum2}) and (\ref{eq:oneloopprim}).
\end{defn2}

\begin{thm}
	The algebra $\mathcal{A}=k[Y^\infty]_h$ with the $\ast_h$ product defined in (\ref{eq:quantumproduct}) and the $\Delta_h$ co-product defined in (\ref{eq:coprodquantum1}) and (\ref{eq:coprodquantum2}) is a Hopf algebra.
\end{thm}
\begin{proof}
	We have already shown the associativity of $\ast_h$. We need to show co-associativity of $\Delta_h$ and the fact that the product and the co-product are homomorphisms \cite{abe2004hopf}. We need also to define the co-unit and discuss the antipode.
	
	The co-associativity of $\Delta_h$ can be shown by induction on the total order of the graph, given by the sum of the genus and the order of the underlying tree. So lets assume that $t^g\in(Y^n)^g$ and that $t^g=t^{g_1}\bridge t^{g_2}$ with $g=g_1+g_2+1$. The case $t^g=t^{g_1}\vee t^{g_2}$ is similar. We have to show that
	\begin{equation}
		\left(\Delta_h\otimes \text{Id}\right)\Delta_h (t^g)=\left(\text{Id}\otimes\Delta_h \right)\Delta_h (t^g).
	\end{equation}
	Using (\ref{eq:coprodquantum2}) the left-hand side gives
	\begin{align}
		&\sum\Delta_h\left((t^{'}_j)^{g^{'}_{1}}\ast_h(t^{'}_k)^{g^{'}_{2}}\right)\otimes\left((t^{''}_{p-j})^{g^{''}_1}\bridge (t^{''}_{q-k})^{g^{''}_2}\right)+\notag\\
		&\sum\left((t^{'}_j)^{g^{'}_1}\ast_h(t^{'}_k)^{g^{'}_2}\right)\otimes\left((t^{''}_{p-j})^{g^{''}_1}\bridge (t^{''}_{q-k})^{g^{''}_2}\right)\otimes 1+t^g\otimes 1\otimes1
	\end{align}
	and the right-hand side is
	\begin{align}
		&\sum\left((t^{'}_j)^{g^{'}_{1}}\ast_h(t^{'}_k)^{g^{'}_{2}}\right)\otimes\Delta_h\left((t^{''}_{p-j})^{g^{''}_1}\bridge (t^{''}_{q-k})^{g^{''}_2}\right)+t^g\otimes 1\otimes1\notag\\
		=&\sum\left((t^{'}_j)^{g^{'}_{1}}\ast_h(t^{'}_k)^{g^{'}_{2}}\right)\otimes\left(\left((t^{''}_{p-j})^{g^{''}_1}\right)^{'}\ast_h \left((t^{''}_{q-k})^{g^{''}_2}\right)^{'}\right)\otimes\notag\\
		&\left(\left((t^{''}_{p-j})^{g^{''}_1}\right)^{''}\bridge \left((t^{''}_{q-k})^{g^{''}_2}\right)^{''}\right)\notag\\
		&+\sum\left((t^{'}_j)^{g^{'}_{1}}\ast_h(t^{'}_k)^{g^{'}_{2}}\right)\otimes\left((t^{''}_{p-j})^{g^{''}_1}\bridge (t^{''}_{q-k})^{g^{''}_2}\right)\otimes 1+t^g\otimes 1\otimes1
	\end{align}
	in a condensed notation for the second iteration of the co-product. Hence it all ends up in proving the equation
	\begin{align}\label{eq:coassociative}
		&\sum\left(\left((t^{'}_j)^{g^{'}_{1}}\right)^{'}\ast_h\left((t^{'}_k)^{g^{'}_{2}}\right)^{'}\right)\otimes\left(\left((t^{'}_j)^{g^{'}_{1}}\right)^{''}\ast_h\left((t^{'}_k)^{g^{'}_{2}}\right)^{''}\right)\otimes\notag\\
		&\left((t^{''}_{p-j})^{g^{''}_1}\bridge (t^{''}_{q-k})^{g^{''}_2}\right)\notag\\
		=&\sum\left((t^{'}_j)^{g^{'}_{1}}\ast_h(t^{'}_k)^{g^{'}_{2}}\right)\otimes\left(\left((t^{''}_{p-j})^{g^{''}_1}\right)^{'}\ast_h \left((t^{''}_{q-k})^{g^{''}_2}\right)^{'}\right)\otimes\notag\\
		&\left(\left((t^{''}_{p-j})^{g^{''}_1}\right)^{''}\bridge \left((t^{''}_{q-k})^{g^{''}_2}\right)^{''}\right)
	\end{align}
	which as mentioned above is proven by induction on the total order of a graph $t^g\in (Y^n)^g$ given by sum of the order $n$ of the underlying tree and genus $g$. For $g=0$ it is the well known co-associativity of the Loday-Ronco Hopf algebra, so let's assume that $g\ge1$. For $g=n=1$ we prove it directly from the definition:
	\begin{align}
		(\Delta_h\otimes\text{Id})\Delta_h(\oneloop)&=(\Delta_h\otimes\text{Id})\left(1\otimes\oneloop+\oneloop\otimes1\right)\notag\\
		&=1\otimes1\otimes\oneloop+1\otimes\oneloop\otimes1+\oneloop\otimes1\otimes1\notag\\
		&=(\text{Id}\otimes\Delta_h)\Delta_h(\oneloop).
	\end{align}

	We assume co-associativity to be true for graphs $s^l=s^{l_1}\vee s^{l_2}$ or $s^l=s^{l_1}\bridge s^{l_2}$ both in $\oplus_{m=1}^{n-1}\oplus_{f=1}^{g}(Y^m)^f$ so that $||s^l||\le n+g-1$, which means that
	\begin{equation}
		(\Delta_h\otimes\text{Id})\Delta_h(s^l)=(\text{Id}\otimes\Delta_h)\Delta_h(s^l).
	\end{equation}
	This is the same as
	\begin{equation}
		(\Delta_h\otimes\text{Id})\sum_{j=1}^a(s^{'}_j)^{l'}\otimes (s_{a-j}^{''})^{l^{''}}=(\text{Id}\otimes\Delta_h)\sum_{j=1}^a(s^{'}_j)^{l^{'}}\otimes (s^{''}_{a-j})^{l^{''}}
	\end{equation}
	or
	\begin{equation}
		\sum\left((s^{'}_j)^{l'}\right)^{'}\otimes\left((s^{'}_j)^{l'}\right)^{''}\otimes (s^{''}_{a-j})^{l^{''}}=\sum (s^{'}_j)^{l^{'}}\otimes \left((s^{''}_{a-j})^{l^{''}}\right)^{'}\otimes \left((s^{''}_{a-j})^{l^{''}}\right)^{''}
	\end{equation}
	with the order of $s^l$ equal to $a$. Considering this relation applied to two graphs $s_1^{l_1},s_2^{l_2}\in\oplus_{m=1}^{n-1}\oplus_{f=1}^{g}(Y^m)^f$ and multiplying term by term we get,
	\begin{multline*}
		\sum\left((s^{'}_{1,j_1})^{l_1^{'}}\right)^{'}\ast_h\left((s^{'}_{2,j_2})^{l_2^{'}}\right)^{'}\otimes\left((s^{'}_{1,j_1})^{l_1^{'}}\right)^{''}\ast_h\left((s^{'}_{2,j_2})^{l_2^{'}}\right)^{''}\otimes\notag\\
		(s^{''}_{1,a_1-j_1})^{l_1^{''}}\ast_h(s^{''}_{2,a_2-j_2})^{l_2^{''}}
	\end{multline*}
	\begin{multline}\label{eq:chooses12}
		=\sum (s^{'}_{1,j_1})^{l_1^{'}}\ast_h(s^{'}_{2,j_2})^{l_2^{'}}\otimes \left((s^{''}_{1,a_1-j_1})^{l_1^{''}}\right)^{'}\ast_h\left((s^{''}_{2,a_2-j_2})^{l_2^{''}}\right)^{'}\otimes\\ \left((s^{''}_{1,a_1-j_1})^{l_1^{''}}\right)^{''}\ast_h\left((s^{''}_{2,a_2-j_2})^{l_2^{''}}\right)^{''}.
	\end{multline}
	Now comparing with (\ref{eq:coassociative}) we almost have the same expression, except that here the last term in the tensor product on the left-hand side of the equation is 
	\begin{equation}\label{eq:lasttensor}
		(s^{''}_{1,a_1-j_1})^{l_1^{''}}\ast_h(s^{''}_{2,a_2-j_2})^{l_2^{''}}
	\end{equation}
	and there the product $\ast_h$ is replaced by $\bridge$, and likewise on the right-hand side. If we could find two graphs $s_1^{l_1},s_2^{l_2}$ such that $(s^{''}_{1,a_1-j_1})^{l_1^{''}}\ast_h(s^{''}_{2,a_2-j_2})^{l_2^{''}}$ would be equal to 
	\begin{equation}\label{eq:lasttensor2}
		(t^{''}_{p-j})^{g^{''}_1}\bridge (t^{''}_{q-k})^{g^{''}_2}
	\end{equation} 
	then we were done. But this is not possible in general since the product $\ast_h$ between two graphs gives a sum of two or more graphs, except in the trivial case where one of the factors is the identity. The most we can expect is $(t^{''}_{p-j})^{g^{''}_1}\bridge (t^{''}_{q-k})^{g^{''}_2}$ to be one of the resulting summands of (\ref{eq:lasttensor}) for appropriate  $s_1^{l_1},s_2^{l_2}$.
	
	So let us search for two graphs of the form $s_1^{l_1}=s_{1,a}^{l_{1,a}}\bridge  s_{1,b}^{l_{1,b}},s_2^{l_2}=s_{2,a}^{l_{2,a}}\vee s_{2,b}^{l_{2,b}}$ with the required conditions. Note that (\ref{eq:lasttensor}) is equal to
	\begin{equation}\label{eq:lasttensor3}
		\left((s_{1})^{l_1})\ast_h((s_{2})^{l_2}\right)^{''}
	\end{equation}on a tensor product and under a sum sign due to the fact that the co-product is an homomorphism, a property we will show in a moment. Also from (\ref{eq:coprodquantum2}) we see that (\ref{eq:lasttensor2}) is equal to the second tensor product factor of the $t^g$ co-product:
	\begin{equation}
		\cdots\otimes\left(t^{g_1}\bridge t^{g_2}\right)^{''} -t^g\otimes 1,
	\end{equation}
	excluding the case $\left(t^{g_1}\bridge t^{g_2}\right)^{''}=1$ and under a sum sign. Let us compute $s_1^{l_1}\ast_h s_2^{l_2}$ explicitly:
	\begin{multline}\label{eq:proofcoasso}
		(s_{1,a}^{l_{1,a}}\bridge s_{1,b}^{l_{1,b}})\ast_h (s_{2,a}^{l_{2,a}}\vee s_{2,b}^{l_{2,b}})=s_{1,a}^{l_{1,a}}\bridge(s_{1,b}^{l_{1,b}}\ast_h s_2^{l_2})+(s_1^{l_1}\ast_h s_{2,a}^{l_{2,a}})\vee s_{2,b}^{l_{2,b}}.
	\end{multline}
	First we take two limit situations: $t^{g_1}=|$ such that $t^g=|\bridge t^{g_2}$, and $t^{g_2}=|$ such that $t^g=t^{g_1}\bridge |$. Since $\|t^g\|=g+n$, in both situations we have that the right or left branches have degree $n-1$ and genus $g-1$ and so satisfy the induction hypothesis. We consider the first case, the second being similar. We take $s_1^1=\oneloop$ and $s_2^{g_2}=t^{g_2}$.
	
	Let us replace our choices in (\ref{eq:chooses12}), that can be rewritten as
	\begin{align}\label{eq:chooses12bis}
		&\sum\left((s_{1}^{l_1}\ast_hs_{2}^{l_2})^{'}\right)^{'}\otimes\left((s_{1}^{l_1}\ast_hs_{2}^{l_2})^{'}\right)^{''}\otimes(s_{1}^{l_1}\ast_hs_{2}^{l_2})^{''}\notag\\
		&=\sum (s_{1}^{l_1}\ast_hs_{2}^{l_2})^{'}\otimes \left((s_{1}^{l_1}\ast_hs_{2}^{l_2})^{''}\right)^{'}\otimes \left((s_{1}^{l_1}\ast_hs_{2}^{l_2})^{''}\right)^{''}.
	\end{align}
	We have
	\begin{align}
		\sum&\left((|\bridge t^{g_2}+(\oneloop\ast_ht_{a}^{g_{2,a}})\vee t_{b}^{g_{2,b}})^{'}\right)^{'}\otimes
		\left((|\bridge t^{g_2}+(\oneloop\ast_ht_{a}^{g_{2,a}})\vee t_{b}^{g_{2,b}})^{'}\right)^{''}\otimes\notag\\
		&\left(|\bridge t^{g_2}+(\oneloop\ast_ht_{a}^{g_{2,a}})\vee t_{b}^{g_{2,b}}\right)^{''}\notag\\
		=\sum&\left(|\bridge t^{g_2}+(\oneloop\ast_ht_{a}^{g_{2,a}})\vee t_{b}^{g_{2,b}}\right)^{'}\otimes
		\left((|\bridge t^{g_2}+(\oneloop\ast_ht_{a}^{g_{2,a}})\vee t_{b}^{g_{2,b}})^{''}\right)^{'}\otimes\notag\\
		&\left((|\bridge t^{g_2}+(\oneloop\ast_ht_{a}^{g_{2,a}})\vee t_{b}^{g_{2,b}})^{''}\right)^{''}.
	\end{align}
	Collecting terms of graphs with the same topology we get several equations. In particular,
	\begin{align}\label{eq:coassociative2}
		\sum&\left((|\bridge t^{g_2})^{'}\right)^{'}\otimes
		\left((|\bridge t^{g_2})^{'}\right)^{''}\otimes
		(|\bridge t^{g_2})^{''}\notag\\
		=\sum&(|\bridge t^{g_2})^{'}\otimes
		\left((|\bridge t^{g_2})^{''}\right)^{'}\otimes
		\left((|\bridge t^{g_2})^{''}\right)^{''}.
	\end{align}
	On the other hand (\ref{eq:coassociative}) gives in this situation
	\begin{align}\label{eq:coassociative3}
		&\sum\left((t^{g_{2}})^{'}\right)^{'}\otimes\left((t^{g_{2}})^{'}\right)^{''}\otimes
		\left(|\bridge (t^{g_{2}})^{''}\right)\notag\\
		=&\sum (t^{g_{2}})^{'}\otimes \left((t^{g_{2}})^{''}\right)^{'}\otimes
		\left(|\bridge \left((t^{g_{2}})^{''}\right)^{''}\right).
	\end{align}
	In view of (\ref{eq:coprodquantum2}) the two equations aren't exactly the same because the first one has the additional terms that occur when in the first co-product $(|\bridge t^{g_2})^{''}=1$ and in this case $(|\bridge t^{g_2})^{'}=(|\bridge t^{g_2})$. So (\ref{eq:coassociative2}) is equal to
	\begin{align}\label{eq:coassociative4}
		\sum\left(|\bridge t^{g_2}\right)^{'}&\otimes\left(|\bridge t^{g_2}\right)^{''}\otimes 1\notag\\
		&+\sum\left(( t^{g_2})^{'}\right)^{'}\otimes
		\left((t^{g_2})^{'}\right)^{''}\otimes
		\left(|\bridge (t^{g_2})^{''}\right)\notag\\
		=\sum\left(|\bridge t^{g_2}\right)^{'}&\otimes\left(|\bridge t^{g_2}\right)^{''}\otimes 1\notag\\
		&+\sum(t^{g_2})^{'}\otimes
		\left((t^{g_2})^{''}\right)^{'}\otimes
		\left(|\bridge \left((t^{g_2})^{''}\right)^{''}\right).
	\end{align}
	and now this is (\ref{eq:coassociative3}).
	
	In the general case we assume that in all graphs $t^g=t^{g_1}\bridge t^{g_2}\in (Y^n)^g$ each branch has $|t^{g_1}|>0$ and $|t^{g_2}|>0$, which makes the order of the other branch smaller than $n-1$ and so both branches have total order $||t^{g_{1,2}}||< n+g-1$, satisfying the induction hypothesis, as does any graph of the type $|\vee t^{g_{1,2}}$ or $|\bridge t^{g_{1,2}}$ or with the left and right branches reversed. We see from eq. (\ref{eq:proofcoasso}) and the discussion leading to it that we must choose $s_1^{l_1}$ and $s_2^{l_2}$ such that $s_{1,a}^{l_1,a}=t^{g_1}$ and $t^{g_2}$ belongs to the sum given by $s_{1,b}^{l_1,b}\ast_h s_2^{l_2}$. We end up on a recursive procedure where now we must choose $s_{1,b}^{l_1,b}$ and $s_2^{l_2}$ as function of $t^{g_2}$ This will stop when one of the factors is the identity. The simplest possibility consistent with the induction hypothesis is to choose $s_{1,b}^{l_1,b}=|$ and $s_2^{l_2}=t^{g_2}$. In this case we have $s_1^{l_1}=t^{g_1}\bridge |, l_1=g_1+1$ and $s_2^{l_2}=t^{g_2}=t_a^{g_{2,a}}\vee t_b^{g_{2,b}}, l_2=g_2$. This is only possible if $t^{g_2}$ doesn't have a middle loop. 
	
	Using (\ref{eq:proofcoasso}) in (\ref{eq:chooses12bis}) with these choices we have
	\begin{align}
		\sum&\left(\left(t^{g_1}\bridge t^{g_2}+(s_1^{l_1}\ast_hs_{2,a}^{l_{2,a}})\vee s_{2,b}^{l_{2,b}}\right)^{'}\right)^{'}\otimes\notag\\
		&\left(\left(t^{g_1}\bridge  t^{g_2}+(s_1^{l_1}\ast_hs_{2,a}^{l_{2,a}})\vee s_{2,b}^{l_{2,b}}\right)^{'}\right)^{''}\otimes\notag\\
		&\left(t^{g_1}\bridge  t^{g_2}+(s_1^{l_1}\ast_hs_{2,a}^{l_{2,a}})\vee s_{2,b}^{l_{2,b}}\right)^{''}\notag\\
		=\sum&\left(t^{g_1}\bridge  t^{g_2}+(s_1^{l_1}\ast_hs_{2,a}^{l_{2,a}})\vee s_{2,b}^{l_{2,b}}\right)^{'}\otimes\notag\\
		&\left(\left(t^{g_1}\bridge  t^{g_2}+(s_1^{l_1}\ast_hs_{2,a}^{l_{2,a}})\vee s_{2,b}^{l_{2,b}}\right)^{''}\right)^{'}\otimes\notag\\
		&\left(\left(t^{g_1}\bridge  t^{g_2}+(s_1^{l_1}\ast_hs_{2,a}^{l_{2,a}})\vee s_{2,b}^{l_{2,b}}\right)^{''}\right)^{''}.
	\end{align}
	After computing the co-product components using (\ref{eq:coprodquantum1}) and  (\ref{eq:coprodquantum2}) we get
	\begin{align}
		\sum&\left(\left((t^{g_1})^{'}\right)^{'}\ast_h\left((t^{g_2})^{'}\right)^{'}+\left(\left((s_1^{l_1}\ast_hs_{2,a}^{l_{2,a}})\vee s_{2,b}^{l_{2,b}}\right)^{'}\right)^{'}\right)\otimes\notag\\
		&\left(\left((t^{g_1})^{'}\right)^{''}\ast_h\left((t^{g_2})^{'}\right)^{''}+\left(\left((s_1^{l_1}\ast_hs_{2,a}^{l_{2,a}})\vee (s_{2,b}^{l_{2,b}}\right)^{'}\right)^{''}\right)\otimes\notag\\
		&\left((t^{g_1})^{''}\bridge ( t^{g_2})^{''}+\left((s_1^{l_1}\ast_hs_{2,a}^{l_{2,a}})\vee s_{2,b}^{l_{2,b}}\right)^{''}\right)\notag
	\end{align}
	\begin{align}
		=\sum&\left((t^{g_1})^{'}\ast_h (t^{g_2})^{'}+\left((s_1^{l_1}\ast_hs_{2,a}^{l_{2,a}})\vee s_{2,b}^{l_{2,b}}\right)^{'}\right)\otimes\notag\\
		&\left(\left((t^{g_1})^{''}\right)^{'}\bridge\left((t^{g_2})^{''}\right)^{'}+\left(\left((s_1^{l_1}\ast_hs_{2,a}^{l_{2,a}})\vee s_{2,b}^{l_{2,b}}\right)^{''}\right)^{'}\right)\otimes\notag\\
		&\left(\left((t^{g_1})^{''}\right)^{''}\bridge  \left((t^{g_2})^{''}\right)^{''}+\left(\left((s_1^{l_1}\ast_hs_{2,a}^{l_{2,a}})\vee s_{2,b}^{l_{2,b}}\right)^{''}\right)^{''}\right).
	\end{align}
	modulo the terms discussed when deducing (\ref{eq:coassociative4}). 
	Hence we get from this last equation several identities relating graphs of different topologies and total degrees, one of which is
	\begin{align}
		\sum&\left((t^{g_1})^{'}\right)^{'}\ast_h\left((t^{g_2})^{'}\right)^{'}
		\otimes\left((t^{g_1})^{'}\right)^{''}\ast_h\left((t^{g_2})^{'}\right)^{''}\otimes
		(t^{g_1})^{''}\bridge ( t^{g_2})^{''}\notag\\
		=\sum&(t^{g_1})^{'}\ast_h (t^{g_2})^{'}\otimes
		\left((t^{g_1})^{''}\right)^{'}\ast_h\left((t^{g_2})^{''}\right)^{'}\otimes
		\left((t^{g_1})^{''}\right)^{''}\bridge  \left((t^{g_2})^{''}\right)^{''}.
	\end{align}
	This is (\ref{eq:coassociative}) in a condensed notation, as required.
	
	If $t^{g_2}$ has a middle loop we may choose instead $s_2^{l_2}=s_{2,a}^{l_{2,a}}\bridge s_{2,b}^{l_{2,b}}$, again $s_1^{l_1}=t^{g_1}\bridge |$ and $s_2^{l_2}=t^{g_2}$ and repeat the proof after (\ref{eq:lasttensor3}).
	
	The fact that the product and the co-product are homomorphisms of a co-algebra and an algebra respectively is equivalent to the equation \cite{abe2004hopf}
	\begin{align}\label{eq:homom}
		\Delta_h(t^{g_1}_1\ast_h t^{g_2}_2)&:=\sum (t^{g_1}_1\ast_h t^{g_2}_2)^{'}\otimes (t^{g_1}_1\ast_h t^{g_2}_2)^{''}\notag\\
		&=\sum (t^{g_1}_1)^{'}\ast_h (t^{g_2}_2)^{'}\otimes (t^{g_1}_1)^{''}\ast_h (t^{g_2}_2)^{''}.
	\end{align}
We will prove this relation by induction on the total order of the product of two graphs and using the definitions (\ref{eq:quantumproduct}), (\ref{eq:coprodquantum1}) and (\ref{eq:coprodquantum2}). For definiteness we will take $t^{g_1}_1=t^{g_{1,a}}_{1,a}\bridge t^{g_{1,b}}_{1,b}, t^{g_2}_12=t^{g_{2,a}}_{2,a}\vee t^{g_{2,b}}_{2,b}$, other cases being treated similarly. 
The first line of (\ref{eq:homom}) is
\begin{align}
&\Delta_h\left(t^{g_{1,a}}_{1,a}\bridge \left(t^{g_{1,b}}_{1,b}\ast_h t^{g_2}_2\right)	+\left(t^{g_1}_1\ast_h t^{g_{2,a}}_{2,a}\right)\vee t^{g_{2,b}}_{2,b}\right)\notag\\
&=\sum\left(t^{g_{1,a}}_{1,a}\right)^{'}\ast_h \left(t^{g_{1,b}}_{1,b}\ast_h t^{g_2}_2\right)^{'}\otimes \left(t^{g_{1,a}}_{1,a}\right)^{''} \bridge\left(t^{g_{1,b}}_{1,b}\ast_ht^{g_2}_2\right)^{''}\notag\\
&+\left(t^{g_{1,a}}_{1,a}\bridge \left(t^{g_{1,b}}_{1,b}\ast_h t^{g_2}_2\right)\right)\otimes 1\notag\\
&+\sum \left(\left(t^{g_1}_1\ast_h t^{g_{2,a}}_{2,a}\right)^{'}\ast_h (t^{g_{2,b}}_{2,b})^{'}\right)\otimes \left(\left(t^{g_1}_1\ast_h t^{g_{2,a}}_{2,a}\right)^{''}\vee (t^{g_{2,b}}_{2,b})^{''}\right)\notag\\
&+\left((t^{g_1}_1\ast_h t^{g_{2,a}}_{2,a})\vee t^{g_{2,b}}_{2,b}\right)\otimes 1.
\end{align}
We now apply the induction hypothesis to products with total order strictly smaller than the order of the initial product, to get
\begin{align}
	&\Delta_h(t^{g_1}_1\ast_h t^{g_2}_2)=\notag\\
	&\sum\Bigl((t^{g_{1,a}}_{1,a})^{'}\ast_h \bigl((t^{g_{1,b}}_{1,b})^{'}\ast_h (t^{g_2}_2)^{'}\bigr)\Bigr)\otimes \Bigl((t^{g_{1,a}}_{1,a})^{''} \bridge\bigl((t^{g_{1,b}}_{1,b})^{''}\ast_h(t^{g_2}_2)^{''}\bigr)\Bigr)\notag\\
	&+\left(t^{g_{1,a}}_{1,a}\bridge \left(t^{g_{1,b}}_{1,b}\ast_h t^{g_2}_2\right)\right)\otimes 1\notag\\
	&+\sum \Bigl(\bigl((t^{g_1}_1)^{'}\ast_h (t^{g_{2,a}}_{2,a})^{'}\bigr)\ast_h (t^{g_{2,b}}_{2,b})^{'}\Bigr)\otimes \Bigl(\bigl((t^{g_1}_1)^{''}\ast_h (t^{g_{2,a}}_{2,a})^{''}\bigr)\vee (t^{g_{2,b}}_{2,b})^{''}\Bigr)\notag\\
	&+\left((t^{g_1}_1\ast_h t^{g_{2,a}}_{2,a})\vee t^{g_{2,b}}_{2,b}\right)\otimes 1.
\end{align}
Using (\ref{eq:coprodquantum1}) and (\ref{eq:coprodquantum2}) again,
\begin{align}
	&\Delta_h(t^{g_1}_1\ast_h t^{g_2}_2)=\notag\\
	&\sum\Biggl(\left(t^{g_{1,a}}_{1,a}\right)^{'}\ast_h \Bigl((t^{g_{1,b}}_{1,b})^{'}\ast_h \bigl((t^{g_2,a}_{2,a})^{'}\ast_h(t^{g_{2,b}}_{2,b})^{'}\bigr)\Bigr)\Biggr)\otimes\notag\\
	& \Biggl(\left(t^{g_{1,a}}_{1,a}\right)^{''} \bridge\Bigl((t^{g_{1,b}}_{1,b})^{''}\ast\bigl((t^{g_{2,a}}_{2,a})^{''}\vee(t^{g_{2,b}}_{2,b})^{''}\bigr)\Bigr)\Biggr)\notag\\
	&+\sum\Bigl(\left(t^{g_{1,a}}_{1,a}\right)^{'}\ast_h \bigl((t^{g_{1,b}}_{1,b})^{'}\ast_h (t^{g_2}_{2})\bigr)\Bigr)\otimes\left((t^{g_{1,a}}_{1,a})^{''}\bridge(t^{g_{1,b}}_{1,b})^{''}\right)\notag\\
	&+\sum \Biggl(\Bigl(\bigl((t^{g_{1,a}}_{1,a})^{'}\ast_h(t^{g_{1,b}}_{1,b})^{'}\bigr)\ast_h (t^{g_{2,a}}_{2,a})^{'}\Bigr)\ast_h (t^{g_{2,b}}_{2,b})^{'}\Biggr)\otimes\notag\\
	& \Biggl(\Bigl(\bigl((t^{g_{1,a}}_{1,a})^{''}\bridge(t^{g_{1,b}}_{1,b})^{''}\bigr)\ast_h (t^{g_{2,a}}_{2,a})^{''}\Bigr)\vee (t^{g_{2,b}}_{2,b})^{''}\Biggr)\notag\\
	&+\sum \Bigl(\bigl((t^{g_{1}}_{1})\ast_h (t^{g_{2,a}}_{2,a})^{'}\bigr)\ast_h (t^{g_{2,b}}_{2,b})^{'}\Bigr)\otimes
	\left((t^{g_{2,a}}_{2,a})^{''}\vee (t^{g_{2,b}}_{2,b})^{''}\right)\notag\\
	&+\left(t^{g_1}_1\ast_h t^{g_{2}}_{2}\right)\otimes 1.
\end{align}
Using associativity and again (\ref{eq:quantumproduct}) we have
\begin{align}
	&\Delta_h(t^{g_1}_1\ast_h t^{g_2}_2)=\notag\\
	&\sum\Bigl(\bigl((t^{g_{1,a}}_{1,a})^{'}\ast_h (t^{g_{1,b}}_{1,b})^{'}\bigr)\ast_h \bigl((t^{g_2,a}_{2,a})^{'}\ast_h(t^{g_{2,b}}_{2,b})^{'}\bigr)\Bigr)\otimes\notag\\
	& \Bigl(\bigl((t^{g_{1,a}}_{1})^{''}\bridge (t^{g_{1,b}}_{1,b})^{''}\bigr)\ast\bigl((t^{g_{2,a}}_{2,a})^{''}\vee(t^{g_{2,b}}_{2,b})^{''}\bigr)\Bigr)\notag\\
	&+\sum\Bigl(\bigl((t^{g_{1,a}}_{1,a})^{'}\ast_h (t^{g_{1,b}}_{1,b})^{'}\bigr)\ast_h t^{g_2}_{2}\Bigr)\otimes\left((t^{g_{1,a}}_{1,a})^{''}\bridge(t^{g_{1,b}}_{1,b})^{''}\right)\notag\\
	&+\sum \Bigl(t^{g_{1}}_{1}\ast_h \bigl((t^{g_{2,a}}_{2,a})^{'}\ast_h (t^{g_{2,b}}_{2,b})^{'}\bigr)\Bigr)\otimes
	\left((t^{g_{2,a}}_{2,a})^{''}\vee (t^{g_{2,b}}_{2,b})^{''}\right)\notag\\
	&+\left(t^{g_1}_1\ast_h t^{g_{2}}_{2}\right)\otimes 1.
\end{align}
This is finally easily seen to be equal to the second line of (\ref{eq:homom}).

The co-unit $\epsilon: \mathcal{A}\rightarrow k$ is given by $\epsilon(|)=1, \epsilon(t^g)=0$ for $t^g\ne|$.

Since the bi-algebra is graded and connected it has an antipode and so is automatically a Hopf algebra \footnote{see \cite{10.2307/1970615}, page 259. Note that what is called a Hopf algebra there is a currently named a bi-algebra and a ``conjugation'' is nowadays called the antipode.}. We limit ourselves to give the expression of the antipode on the primitive elements: $S(\tree)=-\tree,S(\oneloop)=-\oneloop$.
\end{proof}

\subsection{The sub-algebra of topological recursion}Now we take the quotient algebra of regular graphs:
\begin{defn}The algebra $\mathcal{A}_\text{Reg}^h$ of regular graphs is obtained by projecting onto 0 all irregular graphs. 
\begin{equation}
\mathcal{A}_\text{Reg}^h=k[Y^\infty]_h/\{\text{irreg.}\}
\end{equation}
\end{defn}
This will continue to be an associative algebra with the same product (\ref{eq:quantumproduct}) with the rule that every time we try to contract a leaf to produce a loop that is already contracted in another loop we get 0. We have:
\begin{lem}
	In $\mathcal{A}_\text{Reg}^h$, $\oneloop$ is nilpotent.
\end{lem}
\begin{proof}
	$\oneloop\ast_h\oneloop=\onetwolooponeirreg+\twoonelooponeirreg\sim0$ in $\mathcal{A}_\text{Reg}^h$.
\end{proof}
However the co-product given by Definition \ref{def:co-product} is no longer an homomorphism as can be seen from the fact that $\Delta_h\left(\oneloop\ast_h\oneloop\right)=0$ doesn't equal $\Delta_h\left(\oneloop\right)\ast_h\Delta_h\left(\oneloop\right)\neq0$. So $\mathcal{A}_\text{Reg}^h$ is a quotient algebra of $k[Y^\infty]_h$ but not a Hopf algebra, at least with a co-product induced by the one we have been considering. 

We discuss the space of solutions of (\ref{eq:toprec}). In $g=0$ we have a sub-algebra $\mathcal{H}^0_{\text{TopRec}}\subset\mathcal{A}_{\text{Reg}}^h$, that is also a Hopf sub-algebra of Loday-Ronco's Hopf algebra, as the set of solutions of the particular case $g=0$ of Topological Recursion formula\footnote{see (\ref{eq:toprec}).}
\begin{align}\label{eq:toprecgenus0}
	&W_{k+1}^0(p,K)=\sum_{\text{branch points }\alpha}\text{Res}_{q\rightarrow \alpha}K_p(q,\bar{q})\notag\\
	&\left(\sum_{L\cup M=K} W^0_{|L|+1}(q,L)W^{0}_{|M|+1}(\bar{q},M)\right).
\end{align}
In fact, as we have seen in section \ref{sec:toprecg0}, each correlation function $W_{k+1}^0(p,K)$ with associated Euler characteristic $\chi=1-k$ is represented by the sum of all pbt of order $-\chi=k-1$ which in turn is given by the $\ast$ product of $k-1$ factors of $\tree$. In this way the $g=0$ space of solutions $\mathcal{H}^0_{\text{TopRec}}$ is the free algebra on one generator $\tree$. It is a Hopf algebra isomorphic to $k[X]$, the algebra of polynomials in one variable. Hence the ``classical'' nature of the tree expansion is expressed on the commutativity of the Hopf algebra of solutions. 
 
In the case $g>0$ we introduce the generator $\oneloop$ besides $\tree$ and each correlation function $W_{k+1}^g(p,K)$ associated with an Euler characteristic $\chi=1-2g-k$ is again given by the linear combination of all regular graphs of genus $g$ and with underlying trees of order $-\chi=k+2g-1$. This sum is generated by the $\ast_h$ product of $-\chi$ factors of $\tree$ and $\oneloop$ such that there are no two consecutive factors of $\oneloop$ and where the number of loops is equal to the number of $\oneloop$ factors. In other words, the non-commutative sub-algebra $\mathcal{A}^h_{\text{TopRec}}\subset\mathcal{A}_{\text{Reg}}^h$,  presented by two generators $\tree$ and $\oneloop$ and the relation $\oneloop^2=0$, is a module of words with these generators as letters, and these words correspond to sums of regular graphs of order $-\chi$ with loops $g$ of any order  and $k+1$ external legs or labels, up to the relation $\chi=1-2g-k$ and with the number of loops given by the number of factors of the letter $\oneloop$. 
\begin{defn}
	The algebra of graphs $\mathcal{A}^h_{\text{TopRec}}$ that contains solutions of Topological Recursion is a sub-algebra of the quotient algebra of regular graphs $\mathcal{A}_{\text{Reg}}^h$ and is generated by $\tree$ and $\oneloop$ as an non-commutative polynomial algebra over $k$ with the $\ast_h$ product, under the relation $\oneloop\ast_h\oneloop=0$.
\end{defn}
For instance, the word $\oneloop\ast_h\tree\ast_h\oneloop$ is shown in figure \ref{fig:freeword1}.
\begin{figure}[h]
	\begin{tikzpicture}
	\draw[thick] (1,-0.5) -- (1,1) -- (0,2) -- (0.2,1.8)--(0.4,2)--(0.2,1.8)-- (0.5,1.5)--(1,2)--(0.5,1.5)-- (1,1) -- (2,2);\draw (2.5,1) node{\textbf{{\Large $+$}}}; \draw[thick] (4,-0.5) -- (4,1) -- (3,2) -- (4,1) -- (5,2)--(4.8,1.8)--(4.6,2)--(4.8,1.8)--(4.5,1.5)--(4,2);\draw (5.5,1) node{\textbf{{\Large $+$}}}; \draw[thick] (7,-0.5) -- (7,1) -- (6,2) -- (7,1) -- (8,2)--(7.5,1.5)--(7.25,1.75)--(7.5,2)--(7.25,1.75)--(7,2);\draw[thick] (2,2.1) arc (45:135:0.67cm);\draw[thick] (3.95,2.1) arc (30:150:0.58cm);\draw[thick] (6.95,2.1) arc (30:150:0.58cm);\draw[thick] (0.4,2.1) arc (30:150:0.25cm);
	\draw[thick] (5,2.1) arc (30:150:0.25cm);\draw[thick] (8,2.1) arc (30:150:0.25cm);
	\end{tikzpicture}
	\begin{tikzpicture}
	\draw (-0.5,1) node{\textbf{{\Large $+$}}};\draw[thick] (1,-0.5) -- (1,1) -- (0,2) -- (0.35,1.65)--(0.65,2)--(0.35,1.65)-- (1,1) --(1.65,1.65)--(1.35,2)--(1.65,1.65) -- (2,2);\draw (2.5,1);  \draw (3.5,1) node{\textbf{{\Large $+$}}}; \draw[thick] (7,-0.5) -- (7,1) -- (6,2) -- (6.5,1.5)--(6.75,1.75)--(6.5,2)--(6.75,1.75)--(7,2)-- (6.5,1.5) -- (7,1)-- (8,2);
	\draw[thick] (2,2.1) arc (30:150:0.35cm);\draw[thick] (8,2.1) arc (45:135:0.67cm);\draw[thick] (0.6,2.1) arc (30:150:0.35cm);\draw[thick] (6.5,2.1) arc (30:150:0.3cm);
	\end{tikzpicture}
	\caption{Word $\protect\oneloop\ast_h\protect\tree\ast_h\protect\oneloop$ that corresponds to the correlation function $W_1^2(p)$.}\label{fig:freeword1}
\end{figure}
Also the words $\oneloop\ast_h\tree\ast_h\tree,\tree\ast_h\oneloop\ast_h\tree$ and $\tree\ast_h\tree\ast_h\oneloop$ give all fifteen graphs of order 3 with one loop. In all there are 25 graphs of order 3, with 0, 1 and 2 loops.

A generating function $F:k^2\rightarrow \mathcal{A}^h_{\text{TopRec}}$ for these sums of graphs is easily obtained:
\begin{align}
F(a_1,a_2)&=\exp\left(a_1\tree+a_2\oneloop\right)\notag\\
&=1+a_1\tree+a_2\oneloop+\frac{1}{2}\left(a_1^2\tree\ast_h\tree+a_1a_2(\tree\ast_h\oneloop+\oneloop\ast_h\tree)\right)\notag\\
&+\frac{1}{6}\left(a_1^3\tree^3+a_1^2a_2\left(\tree^2\ast_h\oneloop+\tree\ast_h\oneloop\ast_h\tree+\oneloop\ast_h\tree^2\right)\right.\notag\\
&\left.+a_1a_2^2\oneloop\ast_h\tree\ast_h\oneloop\right)\notag+\dots
\end{align}
The module $\textbf{Corr}_{\mathbf{h}}^{\mathbf{(n)}}$ can thus be given a ring structure induced by the algebra $\mathcal{A}^h_{\text{TopRec}}$. A full correlation function $W^n\in\textbf{Corr}_{\mathbf{h}}^{\mathbf{(n)}}$ of Euler characteristic $\chi=-n$ is an expansion in $h$ of coefficients which are correlation functions with $k+1$ variables or punctures and are represented in $\mathcal{A}^h_{\text{TopRec}}$ by products of $p$ factors of $\tree$ and $q$ factors of $\oneloop$ with $\chi=2-2q-k$ and $-\chi=p+q$:
\begin{align}
\psi(W^n)=&\tree\ast\tree\dots\ast\tree + h\oneloop\ast_h\tree\ast_h\tree\dots\ast_h\tree+\notag\\
+&h\tree\ast_h\oneloop\ast_h\tree\dots\ast_h\tree+\dots+h\tree\ast_h\tree\ast_h\dots\ast_h\oneloop\ast_h\tree\notag\\
+&h\tree\ast_h\dots\ast_h\tree\ast_h\oneloop+\dots+h^{q_{\text{max}}}\oneloop\ast_h\tree\ast_h\oneloop\ast_h\dots\oneloop\ast_h\tree\ast_h\oneloop\notag\\
&(\text{ or }\dots+h^{q_{\text{max}}}\oneloop\ast_h\tree\ast_h\oneloop\ast_h\dots\tree\ast_h\oneloop\ast_h\tree\notag\\
&\hspace{2cm}+h^{q_{\text{max}}}\tree\ast_h\oneloop\ast_h\tree\ast_h\dots\ast_h\oneloop\ast_h\tree\ast_h\oneloop).
\end{align}

\section{Cohomology of the quantum algebra}
In \cite{MR1817703} Frabetti studies the simplicial properties of pbt, showing that even though its homology is trivial, they have interesting simplicial features. Namely, face or partial border operators $d_i, i=0\dots n$ that maps $Y^{n}$ into $Y^{n-1}$, are defined by erasing the leaf in position $i$; for all $i,j$ from $0$ to $n$ these maps satisfy the relations
\begin{equation}\label{eq:partialbor}	
	d_id_j=d_{j-1}d_i,i<j.
\end{equation}
From $d_i$ a full face or border operator $d:Y^{n} \to Y^{n-1}$
\begin{equation}
	d=\sum_{i=0}^{n}(-1)^id_i
\end{equation}
is constructed. 

It's easy to show using (\ref{eq:partialbor}) that $d^2=0$ and so $(\oplus_{n=0}^\infty k[Y^n],d)$ is a chain complex. At the same time degeneracy maps $s_i:Y^{n}\to Y^{n+1}$ are defined by bifurcating the leaf at position $i$ for $i=0,\dots, n$, that is, replacing the leaf $\setminus$ or / by $\tree$. Its is shown in \cite{MR1817703} that $(\oplus_{n=0}^\infty k[Y^n],d_i,s_i )$ is an almost simplicial complex, in the sense that the maps $(d_i,s_i)$ satisfy the relations
\begin{equation}
d_is_j=\left\{	
	\begin{array}{lll}
		s_{j-1}d_i, i<j,\\
		Id, i=j,j+1,\\
		s_jd_{i-1}, i>j+1,
	\end{array}\right.
\end{equation}
and also the relation
\begin{equation}
	s_is_j=s_{j+1}s_i,
\end{equation}
except for $i=j$, that is, the relation $s_is_i=s_{i+1}s_i$ is not necessarily verified for all $i$.
\begin{figure}
	\includegraphics{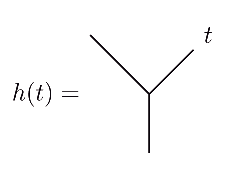}
	\caption{The extra-degeneracy operator.}
	\label{fig:extradegen}
\end{figure}
In addition an extra-degeneracy operator $h:Y^n\to Y^{n+1}$ is defined (see fig. \ref{fig:extradegen}) such that it satisfies the relation
\begin{equation}
	dh+hd=Id.
\end{equation}
This in turn implies that $(\oplus_{n=0}^\infty k[Y^n],d)$ is acyclic. In fact, for a tree $t\in Y^n$ such that $dt=0$ we have from the previous identity that $t=dh(t)$, that is, its homology is trivial.
 
In \cite{1751-8121-48-44-445205} we considered the operator $\prescript{}{i}{\leftrightarrow}_{i+1}$ that contracts the leaves $i$ and $i+1$ in a pbt or a graph of genus $g$ obtained from it, creating a graph of genus 1 or genus $g+1$ respectively. If there isn't a pair of free consecutive leaves in position $i$ and $i+1$ then the result of applying this operator is 0. This contraction operator allows to build a cohomology on $k[Y^\infty]_h$ by defining a differential $d_h:k[(Y^n)^g]\rightarrow k[(Y^n)^{g+1}]$:
\begin{equation}\label{eq:differential}
d_h=\sum_{i=0}^n (-1)^{i+g_i}\prescript{}{i}{\leftrightarrow}_{i+1},
\end{equation}
where $g_i$ is the number of loops before the leaf $i$. 
\begin{lem}
	We have $d_h\circ d_h=0$.
\end{lem}
\begin{proof}
	Let us write $d^{g_i}=(-1)^{g_i}\prescript{}{i}{\leftrightarrow}_{i+1}$. Obviously $g_0=0$ and $g_n=g$ for a graph $t^g\in k [(Y^n)^{g}]$. We have the relations
	\begin{equation}
		\begin{array}{ll}
			&d^{g_i}d^{g_j}=-d^{g_j}d^{g_i},  i<j,\\
			&d^{g_i}d^{g_j}=0, i=j. 
		\end{array}		
	\end{equation}
The first relation is a consequence of the fact that for $j>i$ when acting first with $d^{g_i}$ and then with $d^{g_j}$ we count one more loop in $g_j$ than acting first with $d^{g_j}$ and after with $d^{g_i}$. The second results from the impossibility of producing a second loop between two consecutive leafs that are already contracted to produce a loop.
Then $d_h=\sum_{i=0}^n (-1)^{i}d^{g_i}$ gives
\begin{align}
	d_h^2&=\sum_{i,j=0}^n (-1)^{i}d^{g_i}(-1)^{j}d^{g_j}\notag\\
	&=\sum_{i<j,0}^n (-1)^{i+j}d^{g_i}d^{g_j}+\sum_{i>j,0}^n (-1)^{i+j}d^{g_i}d^{g_j}\notag\\
	&=-\sum_{i>j,0}^n (-1)^{i+j}d^{g_i}d^{g_j}+\sum_{i>j,0}^n (-1)^{i+j}d^{g_i}d^{g_j}\notag\\
	&=0.
\end{align}
\end{proof}
This shows that $(k[(Y^\infty)]_h,d_h)$ is a co-chain complex.

It is clear that $d_h\tree=\oneloop$ and $d_h\oneloop=0$ so that the cohomology is trivial for $n=g=1$. For $n=2,g=1$ things are more interesting:
\begin{equation}\label{eq:donetwo}
d_h\onetwo=\onetwoloopone-\onetwolooptwo,
\end{equation}
which implies that $\onetwoloopone$ and $\onetwolooptwo$ are co-homologous. 
In the same way
\begin{equation}\label{eq:dtwoone}
	d_h\twoone=\twooneloopone-\twoonelooptwo.
\end{equation}	
This implies that
\begin{equation}\label{eq:cohomologous}
	d_h(\tree\ast_h\tree)+\twoonelooptwo+\onetwolooptwo=\onetwoloopone+\twooneloopone.
\end{equation}
However, there is no reason to assume that $\onetwoloopone$, $\onetwolooptwo$, $\twooneloopone$ and $\twoonelooptwo$ are closed. For instance,
\begin{equation}
	d_h\onetwoloopone=\onetwolooponeirreg,
\end{equation}
which is not 0 in $k[Y^\infty]_h$. Hence, let us work on $\mathcal{A}_\text{Reg}^h$ where $\oneloop\ast_h\oneloop=0$ and compute
\begin{equation}\label{eq:oneloopclosed}
	\begin{array}{ll}
		d_h(\tree\ast_h\oneloop)=0\Longrightarrow\tree\ast_h\oneloop=\onetwolooptwo+\twoonelooptwo \text{ is closed;}\\
		d_h(\oneloop\ast_h\tree)=0\Longrightarrow\oneloop\ast_h\tree=\onetwoloopone+\twooneloopone\text{ is closed.}
	\end{array}
\end{equation}
Restricting further to $\mathcal{A}_\text{TopRec}^h$ and writing
\begin{align}
&\mathcal{A}_\text{TopRec}^h=\oplus_{n=0}^\infty\oplus_{g,2g+k-1=n}( \mathcal{A}_\text{TopRec}^h)^{n,g},\\
&Z\left((\mathcal{A}_\text{TopRec}^h)^{n,g}\right)=\left\{t^g\in(\mathcal{A}_\text{TopRec}^h)^{n,g}:d_h(t^g)=0\right\},\\
&B\left((\mathcal{A}_\text{TopRec}^h)^{n,g}\right)=\left\{t^g\in(\mathcal{A}_\text{TopRec}^h)^{n,g}:t^g=d_h(s^{g-1}),s^{g-1}\in(\mathcal{A}_\text{TopRec}^h)^{n,g-1}\right\},
\end{align} 
we see from
\begin{align}
	&(\mathcal{A}_\text{TopRec}^h)^{2,0}=\left\{t\in\mathcal{A}_\text{TopRec}^h:t=a\cdot\tree\ast\tree,a\in k\right\},\\ &(\mathcal{A}_\text{TopRec}^h)^{2,1}=\left\{t^1\in\mathcal{A}_\text{TopRec}^h:t^1=a\cdot\tree\ast_h\oneloop+b\cdot\oneloop\ast_h\tree,a,b\in k\right\}
\end{align}
and also (\ref{eq:cohomologous}) and (\ref{eq:oneloopclosed}) that
\begin{align}
	&Z\left((\mathcal{A}_\text{TopRec}^h)^{2,1}\right)=\left\{t^1\in(\mathcal{A}_\text{TopRec}^h)^{2,1}:t^1=a\cdot\tree\ast_h\oneloop+b\cdot\oneloop\ast_h\tree,a,b\in k\right\},\\
	&B\left((\mathcal{A}_\text{TopRec}^h)^{2,1}\right)=\left\{t^1\in(\mathcal{A}_\text{TopRec}^h)^{2,1}:t^1=a\cdot\left(\oneloop\ast_h\tree-\tree\ast_h\oneloop\right),a\in k\right\},
\end{align} 
so that we have the following
\begin{thm}
\begin{equation}
	H^{2,1}\left((\mathcal{A}_\text{TopRec}^h),k\right)=\left\{ t^1\in (\mathcal{A}_\text{TopRec}^h)^{2,1}:\oneloop\ast_h\tree=\tree\ast_h\oneloop \right\},
\end{equation}
\end{thm}
\begin{cor}
\begin{equation}
	\text{dim}\left\{H^{2,1}\left(\mathcal{A}_\text{TopRec}^h,k\right)\right\}=1.
\end{equation}
\end{cor}

To compute the cohomology in higher degrees we suppose that the use of spectral sequences will be inevitable. Hopefully this will be the subject of a future work.

One last remark: adding (\ref{eq:donetwo}) and (\ref{eq:dtwoone}) we get
\begin{equation}
	d_h\left(\tree\ast\tree\right)=\left[\oneloop,\tree\right],
\end{equation}
with the commutator between two graphs given as usual by
\begin{equation}
	\left[t^{g_1},t^{g_2}\right]=t^{g_1}\ast_ht^{g_2}-t^{g_2}\ast_ht^{g_1}.
\end{equation}
This also shows that $d_h$ satisfies a graded Leibniz rule:
\begin{equation}
	d_h\left(\tree\ast\tree\right)=\left(d_h\tree\right)\ast_h\tree-\tree\ast_h\left(d_h\tree\right),
\end{equation}
a property that we have already used.

\section{Funding}
Partially supported by Funda\c{c}\~{a}o para a Ci\^{e}ncia e a Tecnologia, Portugal through projects PTDC/MAT-PUR/31089/2017 and UID/MAT/04459/2020.

\bibliographystyle{amsplain}
\bibliography{biblio-Loday}
\end{document}